\documentclass[12pt]{article}

\usepackage{amsmath,amssymb,amsthm,amsfonts,mathrsfs,wasysym,latexsym,times,lineno, subfigure,color}
\usepackage{geometry}
\usepackage{graphicx}
\usepackage{subfigure}
\usepackage{epstopdf}
\usepackage{etex}
\newcommand\bes{\begin{eqnarray}}
\newcommand\ees{\end{eqnarray}}

\newtheorem{theorem}{Theorem}[section]
\newtheorem{lemma}[theorem]{Lemma}

\newtheorem{corollary}[theorem]{Corollary}

\theoremstyle{remark}
\newtheorem{remark}[theorem]{Remark}

\numberwithin{equation}{section}

\begin{document}

\title{A Fisher-KPP model with a nonlocal weighted free boundary: analysis of how habitat boundaries expand, balance or shrink}
\author{Chunxi Feng\textsuperscript{1},\; Mark A. Lewis\textsuperscript{2,3},\; Chuncheng Wang\textsuperscript{1},\; Hao Wang\textsuperscript{2}\footnote{Corresponding Author, Email: hao8@ualberta.ca}\\
{\small \textsuperscript{1} Department of Mathematics, Harbin Institute of Technology,\hfill{\ }}\\
\ \ {\small Harbin, Heilongjiang, 150001,  China  \hfill{\ }}\\
{\small \textsuperscript{2} Department of Mathematical and Statistical Sciences,
University of Alberta, \hfill{\ }}\\
\ \ {\small Edmonton, AB T6G 2G1, Canada \hfill {\ }}\\
{\small \textsuperscript{3} Department of Biological Sciences,
University of Alberta, \hfill{\ }}\\
\ \ {\small Edmonton, AB T6G 2E9, Canada \hfill {\ }}}

\date{}
\maketitle
\vspace*{-2ex}

\begin{abstract}

In this paper, we propose a novel free boundary problem to model the movement of single species with a range boundary. The spatial movement and birth/death processes of the species found within the range boundary are assumed to be governed by the classic Fisher-KPP reaction-diffusion equation, while the movement of a free boundary describing the range limit is assumed to be influenced by the weighted total population inside the range boundary and is described by an integro-differential equation. Our free boundary equation is a generalization of the classical Stefan problem that allows for nonlocal influences on the boundary movement so that range expansion and shrinkage are both possible. In this paper we prove that the new model is well posed and possesses steady state. We show that the spreading speed of the range boundary is smaller than that for the equivalent problem with a Stefan condition. This implies that the nonlocal effect of the weighted total population on the boundary movement slows down the spreading speed of the population. While the classical Stefan condition categorizes asymptotic behavior via a spreading-vanishing dichotomy, the new model extends this dichotomy to a spreading-balancing-vanishing trichotomy. We specifically analyze how habitat boundaries expand, balance or shrink. When the model is extended to have two free boundaries, we observe the steady state scenario, asymmetric shifts, or even boundaries moving synchronously in the same direction. These are newly discovered phenomena in the free boundary problems for animal movement.\\

\end{abstract}

\section{Introduction}

Many processes in biology, physics and chemistry can be described by reaction-diffusion equations, see \cite{Murray2003,Cantrell2003,Kot2001,Crank1979}.
Recently, reaction-diffusion equations with free boundary have been used to understand the spreading of species in ecology. For example, the following model was proposed in \cite{Du2010}:
\begin{equation}\label{DuLin}
\begin{aligned}
&u_t=Du_{xx}+f(u),~~~~~t>0,~0<x<h(t),\\
&h'(t)=-\mu u_x(t,h(t)),~~~t>0,\\
&u_x(t,0)=0,~~u(t,h(t))=0,~~~~~~t>0,\\
&h(0)=h_0,~~~u(0,x)=u_0(x),~~0\leq x\leq h_0,
\end{aligned}
\end{equation}
where $D,\mu,h_0>0$, $x=h(t)$ is the moving boundary, and the initial function $u_0(x)$ satisfies
\begin{equation}\label{initial}
u_0(x)\in C^2([0,h_0]),~~~u_0'(0)=u_0(h_0)=0,~~~u_0(x)>0~\text{for}~x\in[0,h_0).
\end{equation}
The equation for the free boundary follows the well-known Stefan condition, which was initially derived for modeling melting ice \cite{Rubinstein1971}. For more mathematical models with free boundary and their analysis, we refer the readers to \cite{Du2015,Chen2000,Lin2007,Wang2017,Mimura1985,Cao2019,Du2017} and references therein.

The speed of melting ice boundary depends on difference of temperatures between ice and water on the boundary in the form of gradient. However, animals may be aware of nonlocal environmental conditions via nonlocal perception (e.g. sight, scents, or sound) and cognition linked to memory, and so should be included in reasonable model formulations for boundary movement. Indeed, if the animals gain a benefit from social interactions, then they would be expected to change their spatial extent based on a spatially weighted kernel that includes aggregation (shrinking spatial extent) over longer spatial scales and repulsion (growing spatial extent) over shorter spatial scales.  Here, long-range aggregation would help ensure the benefits of conspecifics for social interactions, while the short-range repulsion would help ensure that densities do not get too high.  To model this we propose the following equation for movement of the free boundary $h$:
\begin{equation}\label{nonlocalh}
h'(t)=\mu\int_0^{h(t)}u(t,x)w(h(t)-x)dx,~~~t>0,
\end{equation}
 where
\begin{equation}\label{w}
w(h(t)-x)=c_1e^{-\alpha_1(h(t)-x)}-c_2e^{-\alpha_2(h(t)-x)},~~~~c_1>c_2>0,~\alpha_1>\alpha_2>0.
\end{equation}
The equation \eqref{nonlocalh} is formulated based on the assumption that the rate of boundary shift is determined by weighted population inside the range. The weight function $w$ is defined as the subtraction of two decreasing function of the distance between the location $x$ and free boundary $h(t)$. This implies that both expansion and attraction effects are small for the population of animals far away from the boundary. From \eqref{w}, it can be seen that the right hand side of \eqref{nonlocalh} consists of two terms. The first one, $\mu c_1\int_0^{h(t)}u(x,t)e^{-\alpha_1(h(t)-x)}dx$, is the positively weighted total population over shorter spatial scales, providing an expansion effect of the population on the boundary. The second term, $- \mu c_2\int_0^{h(t)}u(x,t)e^{-\alpha_2(h(t)-x)}dx$, is the negatively weighted total population over longer spatial scales, providing a contraction effect on the boundary.

In our model the population has a range boundary, within which it grows and thrives and outside of which it dies quickly.  Individuals are free to cross over the range boundary from the interior of the range to the exterior of the range, but once they leave, they die.  There are many possible models for the movement of the range boundary itself.  Here we consider a social population, where individuals accrue benefits from longer range interactions with other individuals, but suffer when there is too much crowding.  Using this idea, it is assumed that movement of the range boundary is based on a balance between short and long range interactions as experienced by individuals at the range boundary. Mathematically this is given by the weighted population density experienced by individuals at the range boundary and is written as the population density functional
\begin{equation}\label{Jdef}
J[u(t,x),h(t)] = \int_0^{h(t)}u(t,x)w(h(t)-x)dx.
\end{equation}
When $J>0$ the short-range crowding terms experienced by individuals at the boundary exceed the beneficial long-range interaction terms and so range is enlarged ($h'(t)>0$). When $J<0$ the  long-range interaction terms experienced by individuals at the boundary exceed the short-range terms and so range is shrunk ($h'(t)<0$). Finally, when short-range and long-range terms balance at the boundary ($J=0$) the range boundary remains stationary ($h'(t)=0$). The velocity of range boundary movement is assumed to be proportional to the deviation of weighted population density function (\ref{Jdef}), with proportionality parameter $\mu$, leading to equation (\ref{nonlocalh}). When animals leave the range they are assumed to die quickly.  In our model, we assume that the death rate outside the range is large so that the population is effectively zero outside the range.

The requirements of $c_1,c_2,\alpha_1,\alpha_2$ in \eqref{w} ensures that the expansion effect dominates at the local spatial scale for small initial habitat size $h_0$. In fact, if $h_0\leq(\alpha_1-\alpha_2)^{-1}\ln (c_1/c_2)$, then $w(h_0-x)\geq0$ for $x\in(0,h_0)$, and therefore, $h^\prime(t)>0$ for sufficiently small $t>0$. However, the boundary of habitat may not expand on large scales, whenever $h_0>(\alpha_1-\alpha_2)^{-1}\ln (c_1/c_2)$. In this case, one can always choose a specific initial value $u_0(x)$ such that $h'(0)<0$, leading to the contraction of the boundary, since $w(h_0-x)$ changes the sign for $x\in(0,h_0)$. With this new model, animals can expand or shrink their range as desired.

Now, given a range boundary $h(t)$, the question arises as to how animals undergoing a random walk on the interior of the range will move relative to the range boundary.  Possibilities range from turning back into the range when confronted with the range boundary (Neumann boundary condition) to being more likely to turn back than from random chance alone (Robin boundary condition) to continuing on the random walk, effectively ignoring the range boundary (Dirichlet boundary condition).  All of these may be reasonable under different assumptions.  In this paper we focus on the last possibility, leading to the Dirichlet boundary $u(t,h(t))=0$ in condition given in equation \eqref{DuLin}, whilst leaving the other possibilities for future analysis.

We remark that the Stefan condition in \eqref{DuLin} can be viewed as a special case of \eqref{nonlocalh} in our model. We show this by constructing a sequence of weighting functions $w_n$ in equation \eqref{w} whose limit approximates the derivative of the delta function, and thus after careful integration by parts before taking the limit, yields the Stefan condition in equation \eqref{nonlocalh}. We choose $c_1=n^3$, $c_2=n^2$, $\alpha_1=n^2$, $\alpha_2=n$ and rewrite $w(x)$ as $w_n(x)$ and $h(t)$ as $h_n(t)$ for $n\in\mathbb{N}$. Then, \eqref{nonlocalh} becomes
$$
h_n'(t)=\mu\int_0^{h_n(t)}u(t,x)w_n(h_n(t)-x)dx=\mu\int_0^{h_n(t)}u(t,h_n(t)-x)w_n(x)dx.
$$
As $n\rightarrow \infty$, we have
\begin{equation}\label{tostefan}
h'(t)=\lim_{n\rightarrow\infty}h_n'(t)=\lim_{n\rightarrow\infty}\mu\int_0^{h_n(t)}u(t,h_n(t)-x)w_n(x)dx=-\mu u_x(t,h(t)),
\end{equation}
which is the Stefan condition as given in \eqref{DuLin}. The proof of \eqref{tostefan} is provided in Appendix.

In this paper, we study the following system of equations:
\begin{equation}\label{model1}
\begin{aligned}
&u_t=Du_{xx}+f(x,u),~~~~~t>0,~0<x<h(t),\\
&h'(t)=\mu\int_0^{h(t)}u(t,x)w(h(t)-x)dx,~~~t>0,\\
&u_x(t,0)=0,~~u(t,h(t))=0,~~~~~~t>0,\\
&h(0)=h_0,~~~u(0,x)=u_0(x),~~0\leq x\leq h_0,
\end{aligned}
\end{equation}
where $u_0(x)$ satisfies \eqref{initial} and $w$ is given by \eqref{w}.
Using the standard boundary stretching technique and contraction mapping theorem, the local existence and uniqueness of the solution can be proved. In \cite{Du2010}, it has been proved that the spreading-vanishing dichotomy holds for \eqref{DuLin}; that is, the population either establishes itself with an expanding range boundary, or vanishes eventually. In particular, there is a critical $\hat{h}$ such that, for $h_0<\hat{h}$, whether establishment with an expanding range boundary occurs is determined by the initial population density $u_0(x)$.  For \eqref{model1}, we numerically show that, it could possess another dynamical behavior for proper choices of parameters, besides spreading and vanishing, that is, there exists $(h^*,\bar{u}(x))$ such that $h(t)\rightarrow h^*$ and $u(t,x)\rightarrow \bar{u}(x)$ as $t\rightarrow\infty$, where $h^*$ is a positive constant and $\bar{u}(x)$ is the nonconstant steady state of \eqref{model1}. The pair $(h^*,\bar{u}(x))$ is referred as the steady state of \eqref{model1}, and the dynamics associated with the stable steady state is called balancing in the context. Therefore, spreading-balancing-vanishing trichotomy is expected to be held under certain conditions for the proposed model. For instance, if we use $c_1$ as the varying parameter, the dynamics of \eqref{model1} change from vanishing to balancing and then to spreading, as $c_1$ increases. Moreover, spreading-vanishing dichotomy is also detected, if $h_0$ is chosen as the varying parameter. Furthermore, for \eqref{model1} we can find a critical $\hat{h}$, such that the occurrence of spreading is determined by the initial condition $u_0(x)$ when $h_0<\hat{h}$. In addition, when the spreading occurs, the speed in the nonlocal free boundary model \eqref{model1} is slowed down compared to the speed in \eqref{DuLin}.

We also consider \eqref{model1} with two sides of free boundary. The double front spreading problem of \eqref{DuLin} with the Stefan conditions has been studied in \cite{Du2010}. It has been shown in \cite{Du2010} that most of the results for \eqref{DuLin} with one side of free boundary are still valid for the the situation of double-front free boundary problem. For more free boundary problems with two boundaries, we refer the readers to \cite{Du2014,Wang2017,Wang20172,Cao2019,Wu2015,Wu20152} and references therein. For \eqref{model1} with two boundaries, we numerically illustrate that the dynamics of spreading-balancing-vanishing trichotomy will also take place. The spreading dynamics of \eqref{model1} with two free boundaries include the diffusion of population with both range boundaries moving in one direction, and the movement of two boundaries in the opposite directions with different speeds. This asymmetric movement is new for a free boundary problem without advection in the reaction-diffusion equation because advection was the reason for asymmetric expansion in Stefan problems \cite{GuLouZhou,Gu2014,Gu2015,Kaneko2015}. The increment of $c_1$ or $r$ will speed up the movement of the free boundary as expected. Aside from asymmetric movement, we observe the steady state scenario or even boundaries moving synchronously in the same direction.

The paper is organized as follows. Section 2 focuses on the global existence and uniqueness of solutions to \eqref{model1}. The existence of steady state is also proved in this part, which can lead to the spreading-balancing-vanishing trichotomy. In Sections 3 and 4, numerical simulations illustrate various insightful and novel dynamics of \eqref{model1} with one or two free boundaries. Section 5 concludes this research and suggests future directions. In Appendix, we prove that the Stefan condition is a special case of our free boundary condition, and provide an estimation for the steady state of an elliptic equation.

\section{Well-posedness}

For the well-posedness of the new free boundary problem \eqref{model1}, we assume that
\begin{itemize}
\item[$(H1)$] $f(x,u)$ is locally Lipschitz continuous in $x$, i.e, for any given $\delta,~l$, there exists a constant $L(\delta,l)$  satisfying
$$\left| {f\left( {x,u} \right) - f\left( {y,u} \right)} \right| \le L\left( {\delta ,l} \right)\left| {x - y} \right|$$
for $x,y \in [0,\delta ]$ and $u \in [0,l]$.
\item[$(H2)$] There exists a constant $L_1$ such that
$$
|f(\cdot,u)-f(\cdot,v)|\leq L_1|u-v|.
$$
\end{itemize}

Let $\Lambda=\{c_1,c_2,\alpha_1,\alpha_2,h_0,h^\prime(0),\|u_0\|_{C^2([0,h_0])}\}$. For any $0<T<\infty$, denote
$$
D_T=(0,T]\times(0,h(T)].
$$
We first prove the local existence and uniqueness results of \eqref{model1} by the contraction mapping theorem.

\begin{theorem}\label{eu}
Assume that $(H1)$ and $(H2)$ are satisfied. Then,
for any $\alpha\in(0,1)$ and $u_0$ satisfying \eqref{initial}, there exists a $T>0$ such that \eqref{model1} has a unique solution $(u,h)$ defined on $[0,T]$. Moreover,
$$
u\in C^{(1+\alpha)/2,1+\alpha}(D_T),~~~~h\in C^{1+\alpha/2}([0,T]).
$$
\end{theorem}
\begin{proof}
Let
$$
y=x/h(t),~~ v(t,y)=u(t,h(t)y),
$$
which changes the free boundary $x=h(t)$ to fixed boundary $y=1$. Direct calculations show that
$$
u_x=h^{-1}(t)v_y,~~~u_{xx}=h^{-2}(t)v_{yy},~~~u_t=v_t- yh'(t)u_x.
$$
Then, \eqref{model1} becomes
\begin{equation}\label{ll}
\begin{cases}
v_t - a(t)v_{yy} - b(t)yv_y = f(h(t)y,v),&t > 0,~~0 < y < 1,\\
v_y(t,0) = 0,v(t,1) = 0,&t > 0,\\
v(0,y) = u_0(h_0y),&0 \le y \le 1,
\end{cases}
\end{equation}
and
\begin{equation}\label{model3}
h'(t) =\mu \int_0^1 v(t,y)w(h(t)-h(t) y)h(t)dy,~~t>0,~~h(0)=h_0,
\end{equation}
where
$$
a(t)=Dh^{-2}(t),~b(t)=h'(t)/h(t).
$$

Let $T_1=\min \{ 1,\frac{h_0}{{2\left( {1 + h_1} \right)}}\} $. For $0<T\leq T_1$, denote
$$\Omega_T=\{h\in C^{1}([0,T]):h(0)=h_0,~~h'(0)=h_1,~~\| h'-h_1\|_{C([0,T])}\leq1\}.$$
Then $\Omega_T $ is a bounded and closed convex set of $C^{1}([0,T])$. For $h\in\Omega_T$, we have
$$\left|h(t)-h_0\right|\leq T\|h'\|_{C([0,T])}\leq T_1(1+h_1)\leq\frac{h_0}{2},~~~\forall t\in[0,T_1].$$
Applying $L^{p}$ theory and Sobolev embedding theorem, we know that, for any given $h\in\Omega_T$, there exists $0<T^*\leq T_1$, depending only on $\Lambda$ and the bound of $f$ on $[0,T_1]\times[0,\frac{3h_0}{2}]\times[0,\| u\|_\infty]$, such that \eqref{ll}
admits a unique solution $\overline{v}\in C^{(1+\alpha)/2,1+\alpha}(\Delta_{T^*})$ and
\begin{equation}\label{1}
\left\| \overline{v} \right\|_{C^{(1+\alpha)/2,1+\alpha}(\Delta_{T^*})}\le C_1(\Lambda,T^*,{T^*}^{-1})
\end{equation}
where $\Delta_{T^*}=[0,T^*]\times[0,1]$. Since the bound of $f$ on $[0,T_1]\times[0,\frac{3h_0}{2}]\times[0,\| u\|_\infty]$ depends only on $\Lambda$ , we may take $C_1(\Lambda,T^*,{T^*}^{-1})$ as $C_1(\Lambda)$, since $T^*$ depends only on $\Lambda$. So
$$
\| \overline{v} \|_{C^{(1+\alpha)/2,1+\alpha}(\Delta_{T^*})}\le C_1(\Lambda).
$$
Therefore, for $0<T\leq T^*$, the unique solution of \eqref{ll} satisfies
$$
\| \overline{v} \|_{C^{(1+\alpha)/2,1+\alpha}(\Delta_{T})}\le \| \overline{v} \|_{C^{(1+\alpha)/2,1+\alpha}(\Delta_{T^*})}\le C_1(\Lambda).
$$
In addition, it follows from the positivity lemma that $\overline{v}>0$ on $(0,T]\times[0,1]$. For such $\overline{v}$, we consider
$$
\overline{h}(t) =h_0+ \int_0^t \mu \int_0^1\overline{v}(\tau,y)w(h(t)-h(t)y)h(\tau)dyd\tau.
$$
Then,
$${\overline h ^\prime }\left( t \right) = \mu \int_0^1 {\overline v (t,y)} w(h(t) -h(t) y)h(t)dy,~~\overline h(0)=h_0,~~\overline h^\prime(0)=h_1,$$
and therefore, $\overline h^\prime\in C^{\alpha/2}([0,T])$ and
\begin{equation}\label{l2}
\|{\overline h }^\prime \|_{C^{\alpha/2}([0,T])} \le {C_2},
\end{equation}
where $C_2$ is a constant depending on $C_1$ and $\mu$.

Now, we define $\mathcal{F}:\Omega_T\rightarrow C^{1}([0,1])$ by
$$\mathcal{F}(h)=\overline h.$$
Clearly, $h\in\Omega_T$ is a fixed point of $\mathcal{F}$ if and only if $(v,h)$ solves \eqref{ll} and \eqref{model3}.
By \eqref{l2}, we have
$$\|{\overline h}^\prime-h_1\|_{C([0,T])}\le \|{\overline h}^\prime\|_{C^{\alpha /2}([0,T])}{T^{\alpha /2}} \le {C_2}{T^{\alpha /2}}.$$
Therefore, if we choose $T$ so small that $T\leq \min\{1,h_0/2(1+h_1),{C_2}^{-2/\alpha}\}$, then $\mathcal{F} $ maps $\Omega_T$ into itself. It remains to show
$\mathcal{F} $ is a contraction mapping for $0<T\ll1$. Let $(v_i, h_i)\in C(D_T)\times \Omega_T(i=1,2)$. For each $h_i\in\Omega_T$, $i=1,2$, we can get $\overline v_i$ by solving \eqref{ll}. Denote $\overline h_i=\mathcal{F}(h_i)$. Assume that $\|\overline v_i-v_0\|_{C(\Delta_T)}\leq1(i=1,2)$. It then follows from \eqref{1} and \eqref{l2} that
$$ {\left\| {\overline v }_i \right\|_{{C^{\left( {1 + \alpha } \right)/2,1+ \alpha }}\left( {{\Delta _T}} \right)}} \le {C_1(\Lambda)}, ~~\|{\overline h }^\prime_i\|_{C^{\alpha/2}([0,T])} \le {C_2}$$
for $i=1,2$. Set $V={\overline v}_1-{\overline v}_2,~h=h_1-h_2,~\overline h=\overline h_1-\overline h_2$. Then
\begin{equation}\label{l3}
\begin{cases}
V_t-a_1(t)V_{yy}-{b_1}(t)yV_y-\rho(t,y)V=[a_1(t)-a_2(t)]{\overline v}_{2yy},\\~~~~~~+[b_1(t)-b_2(t)]{\overline v}_{2y}+\beta(t,y)yh,~~0<T\leq T,~~0<y<1,\\
V_y(t,0)=0,~~V(t,1)=0,~~~~0<t\leq T,\\
V(0,y)=0,0<y<1,
\end{cases}
\end{equation}
and
\begin{equation}\label{l4}
\begin{cases}
{\overline h'(t)}=\mu\int_0^1{\overline v_1}(t,y)w_1(t,y)h_1dy-\mu\int_0^1{\overline v_2}(t,y)w_2(t,y)h_2dy,~~0<t\leq T, \\
{\overline h(0)}=0,
\end{cases}
\end{equation}
where ${a_i}(t)=D{h_i}^{-2}(t),~{b_i}(t)={h'_i}(t)/{h_i}(t),~w_i(t,y)=w(h_i(t)-h_i(t)y),~i=1,2$, and
$$\rho(t,y)=\frac{f(h_1(t)y,{\overline v}_1)-f(h_1(t)y,{\overline v}_2)}{{\overline v}_1-{\overline v}_2},$$
$$\beta(t,y)=\frac{f(h_1(t)y,{\overline v}_2)-f(h_2(t)y,{\overline v}_2)}{(h_1(t)-h_2(t))y}.$$
By the assumption on $f$, we see that $\rho,\beta\in C(\Delta_T)$, and $\|\rho\|_{C(\Delta_T)}, \|\beta\|_{C(\Delta_T)}$ depend only on $h_0$ and $\|u_0\|_{C(\Delta_T)}$. Recall that $|h(t)-h_0|\leq h_0/2,~t\in[0,T]$. Applying the $L^p$ theory to \eqref{l3}, we obtain
\begin{equation}\label{l5}
\begin{aligned}
\|V\|_{W_p^{1,2}}&\leq C_3(\|({a_1}-{a_2}){\overline v}_{2yy}\|_{L^p(\Delta_T)}+\|({b_1}-{b_2})y{\overline v}_{2y}\|_{L^p(\Delta_T)}+\|\beta yh\|_{L^p(\Delta_T)})\\&\leq C_4\|h\|_{C^1([0,T])},
\end{aligned}
\end{equation}
where $C_3$ depends on $\Lambda$ , and $C_4$ depends on $C_3$.
From \eqref{l4}, we have
\begin{equation}\label{l6}
\begin{aligned}
{[{\overline h }^\prime ]_{C^{\alpha /2}([0,T])}}&\leq\mu\left[\int_0^1{\overline v}_1\tilde{w}h_1dy\right]_{C^{\alpha/2}([0,T])}+\mu\left[\int_0^1Vw_2h_1dy\right]_{C^{\alpha/2}([0,T])}\\
&\quad+\mu\left[\int_0^1\overline v_2w_2hdy\right]_{C^{\alpha/2}([0,T])},
\end{aligned}
\end{equation}
where $\tilde{w}=w_1-w_2$, and $[\cdot]_{C^{\alpha/2}}$ is the Holder semi-norm. Following the proof of the Theorem 1.1 in \cite{Wang2019}, we can show that
\begin{equation}\label{l7}
\begin{aligned}
{[V]}_{C^{\alpha,\alpha/2}(\Delta_T)}\leq C\|V\|_{W_p^{1,2}(\Delta_T)}.
\end{aligned}
\end{equation}
where $C$ is independent of $T^{-1}$.
Remember that $ {\left\| {\overline v }_i \right\|_{{C^{\left( {1 + \alpha } \right)/2,1+ \alpha }}\left( {{\Delta _T}} \right)}} \le {C_1(\Lambda)}$. Combining with \eqref{l5} and \eqref{l6}, we have
\begin{equation}
\begin{aligned}
{[{\overline h }^\prime ]_{{C^{\alpha /2}}([0,T])}} \le C_5(\Lambda)\|h\|_{{C^1}([0,T])}.
\end{aligned}
\end{equation}
Noticing that $\overline h(0)=\overline h'(0)=0$, it is deduced that
$$ \|\overline h\|_{C^1([0,T])}\leq 2T^{\alpha/2}\|{\overline h }^\prime \|_{C^{\alpha /2}([0,T])}  \le 2C_5(\Lambda)T^{\alpha/2}\|h\|_{{C^1}([0,T])}.$$
Hence, if we take
$$T\leq \min\{1,h_0/2(1+h_1),{C_2}^{-2/\alpha},(2C_5)^{-\alpha/2}\},$$
then
$$\|\overline h_1-\overline h_2\|_{C^1([0,1])}\leq \frac{1}{2}\|h_1-h_2\|_{C^1([0,1])}.$$
This shows that, for such a choice of $T$, $\mathcal F$ is a contraction mapping on $\Omega_T$ and thus admits a unique fixed point in $\Omega_T$. Therefore, \eqref{model1} has a unique solution $(u,h)$ satisfying $u\in C^{(1+\alpha)/2,1+\alpha}(D_T),h\in C^{1+\alpha/2}([0,T])$.
\end{proof}
To show that the local solution proved in Theorem \eqref{eu} can be globally extended to all $t>0$, we firstly give the following lemma. In what follows, we shall take $f(u)$ as $u(a-u)$ for the sake of making some specific analysis.

\begin{lemma}\label{lem}
Assume that $(H1)$ and $(H2)$ hold. Let $T>0$, $(u,h)$ be a solution to problem \eqref{model1}. Then there exist constants $K$ and $M$  independent of $T$, such that
$$0<u(t,x)\leq M,~~~ h'(t)\leq K~~~~~~~t\in[0,T),~~x\in[0,h(t)).$$
\end{lemma}
\begin{proof}
By the maximum principle, we infer $u(t,x)>0$ for $t\in(0,T)$ and $x\in(0,h(t))$. \\
Let $\overline u(t)=\frac{ae^{at}}{e^{at}-1+\frac{a}{\|u_0\|_\infty}}$ be the solution to
\begin{equation}
\begin{cases}
\overline u_t=\overline u(a-\overline u),~~&0<t<T,\\
\overline u_x(t,0)=0,~\overline u(t)>0,~~&0<t<T,\\
\overline u(0)=\|u_0\|_\infty,~~&t=0.
\end{cases}
\end{equation}
According to the comparison principle, we deduce that $u(t,x)\leq\overline u$ for $t\in(0,T)$ and $x\in[0,h(t)]$. As $\sup_{t\geq 0}\overline u(t)\leq M$, thus $u(t,x)\leq M$.

It remains to prove that $h'(t)\leq K$ for $t\in(0,T)$. Since $u(t,x)\leq M$, we have
\begin{equation}\nonumber
\begin{aligned}
h^\prime(t)&=\mu\int_0^{h(t)}u(t,x)(c_1e^{-\alpha_1(h(t)-x)}-c_2e^{-\alpha_2(h(t)-x)})dx\\
&\leq\mu\int_0^{h(t)}M(c_1e^{-\alpha_1(h(t)-x)}+c_2e^{-\alpha_2(h(t)-x)})dx\\
&\leq\frac{\mu c_1M}{\alpha_1}(1-e^{-\alpha_1(h(t))})+\frac{\mu c_2M}{\alpha_2}(1-e^{-\alpha_2(h(t))})\\
&\leq\mu M(\frac{c_1}{\alpha_1}+\frac{c_2}{\alpha_2})~:=K.
\end{aligned}
\end{equation}
\end{proof}
\begin{theorem}Assume that $(H1)$ and $(H2)$ are satisfied.
The problem \eqref{model1} has a unique global solution existing for all $t\in(0,\infty)$.
\end{theorem}
\begin{proof}
Assume that $(0,T_{max})$
is the maximal time interval of the solution of problem \eqref{model1}. According to Theorem \ref{eu}, we know $T_{max}>0$. Therefore, it remains to prove that $T_{max}=\infty$. If it does not hold, i.e $T_{max}<\infty$, we choose $t_n\in(0,T_{max})$ and $t_n\rightarrow T_{max}$ as $n\rightarrow\infty$. Let $t_n$ be the initial time, by Lemma \ref{lem}, there exist $M$ and $K$ independent of $t_n$ such that
$$0<u(x,t_n)\leq M,~~~ h'(t_n)\leq K,~~~h_0\leq h(t_n)\leq h_0+Kt_n.$$
Repeating the process of proving Theorem \ref{eu}, we can obtain a constant $\tau>0$ that depends only on $h(t_n)$, $h'(t_n)$, $\|u(t_n,\cdot)\|_{C^2{(0,h(t_n))}}$ and $\|f\|_{L^\infty}$ such that  \eqref{model1} has a unique solution $u(t_n,x)$ in $[t_n,t_n+\tau]$. Thus we can extend the solution uniquely to $(0,t_n+\tau]$ by the uniqueness of solution. However, $t_n+\tau>T_{max}$ when $n$ is sufficiently large since $t_n\rightarrow T_{max}$ as $n\rightarrow\infty$. But this contradicts the assumption of $T_{max}$.
\end{proof}
Assume that $h_\infty=\lim_{t\rightarrow\infty}h(t)$ exists and $h_\infty>h(t)$ for $t\geq0$, we obtain following lemmas.
\begin{lemma}
If $h_\infty\leq\frac{\pi}{2}\sqrt\frac{D}{a}$, then $\lim_{t\rightarrow\infty}\|u(t,\cdot)\|_{C([0,h(t)])}=0$.
\end{lemma}
\begin{proof}
Let $\overline u(t,x)$ be the unique solution of the problem
\begin{equation}\nonumber
\begin{cases}
\overline u_t-D\overline u_{xx}=\overline u(a-\overline u),~~~&t>0,~0<x<h_\infty,\\
\overline u_x(t,0)=0,~~\overline u(t,h_\infty)=0,~~~&t>0,\\
\overline u(0,x)=\tilde u_0(x),~~&0<x<h_\infty,
\end{cases}
\end{equation}
where $\tilde u_0(x)=u_0(x)$ if $0\leq x\leq h_0$, and $\tilde u_0(x)=0$ if $x\geq h_0$.

It follows from comparison principle that $0\leq u(t,x)\leq \overline u(t,x)$ for $t>0$ and $x\in[0,h(t)]$. Since $h_\infty\leq\frac{\pi}{2}\sqrt\frac{D}{a}$ thus $a<D(\frac{\pi}{2h_\infty})^2$ and we have $\overline u(t,x)\rightarrow0$ as $t\rightarrow\infty$ uniformly for $x\in[0,h_\infty]$. Therefore $\lim_{t\rightarrow\infty}\|u(t,\cdot)\|_{C([0,h(t)])}=0$.
\end{proof}

\begin{lemma}
If $h_\infty=\infty$, then $\lim_{t\rightarrow\infty}u(t,x)=a$ in any bounded subset of $[0,\infty)$.
\end{lemma}
\begin{proof}
According to Lemma \ref{lem}, we have $u\leq\overline u=\frac{ae^{at}}{e^{at}-1+\frac{a}{\|u_0\|_\infty}}$.
 Since $\lim_{t\rightarrow\infty}\overline u=a$, it follows that $\lim \sup_{t\rightarrow\infty}u(t,x)\leq a$
uniformly for $x\in[0,\infty)$.

 Let $l>\max\{h_0,\frac{\pi}{2}\sqrt{\frac{D}{a}}\}$. As $h_\infty=\infty$, there exists $T\gg1$ such that $h(t)>l$ for $t\geq T$. Let $\phi$ be the positive eigenfunction of
\begin{equation}\label{lambda1}
\begin{cases}
-Dw^{\prime\prime}=\lambda w,~~0<x<l,\\
w_x(0)=0,~w(l)=0,
\end{cases}
\end{equation}
corresponding to the first eigenvalue $\lambda_1$. Choose $0<\delta\leq1$ such that $u(T,x)\geq\delta\phi(x)$ in $[0,l]$. Let $u^l$ be the unique solution to
\begin{equation}
\begin{cases}
u^l_t=Du^l_{xx}+ u^l(a-u^l),~~&t\geq T,0<x<l,\\
u^l_x(t,0)=0,~ u^l(t,l)=0,~~&t\geq T,\\
u^l(T,x)=\delta\phi(x),~~&0\leq x\leq l.
\end{cases}
\end{equation}
Then $u\geq u^l$ in $[T,\infty)\times[0,l]$. Since $a>D{(\frac{\pi}{2l})}^2$, it follows that $u^l(t,x)\rightarrow u^{\star}$ as $t\rightarrow\infty$, where $u^{\star}$ is the unique solution of
\begin{equation}\label{stst}
\begin{cases}
-Du^{\star}_{xx}=u^{\star}(a-u^{\star}),~~0<x<l,\\
u^{\star}_x(0)=0,~~u^{\star}(l)=0.
\end{cases}
\end{equation}
Therefore, $\liminf_{t\rightarrow\infty}u(t,x)\geq u^{\star}(x)$. From \cite{Cantrell2003} (corollary 3.4) or \cite{Du2010}, we have $u^{\star}(x)\rightarrow a$ as $l\rightarrow\infty$ uniformly in compact subset of $[0,\infty)$, and therefore,
$$\liminf_{t\rightarrow\infty}u(t,x)\geq a$$
uniformly in any compact subset of $[0,\infty)$. Combining with $\limsup_{t\rightarrow\infty}u(t,x)\leq a,$ we complete the proof.
\end{proof}

We say $(u^*,h^*)$ is a steady state of  the problem \eqref{model1} provided that it satisfies
\begin{equation}
\begin{aligned}
&Du^*_{xx}+f(x,u^*)=0,~~~~~~0<x<h^*,\\
&\mu\int_0^{h^*}u^* w(h^*-x)dx=0,\\
&u^*_x(0)=0,~~u^*(h^*)=0.~~~~~\\
\end{aligned}
\end{equation}

\begin{lemma}\label{main}
Assume $\frac{\pi}{2}\sqrt\frac{D}{a}<h_\infty\leq\infty$. If $a$ is large enough, and
\begin{equation}\label{cond}
\frac{c_1}{\alpha_1}<\frac{c_2}{\alpha_2}
\end{equation}
holds, then the problem \eqref{model1} with $f(u)=u(a-u)$ has a steady state $(u^*,h^*)$.
\end{lemma}
\begin{proof}
Since $h_\infty>\frac{\pi}{2}\sqrt\frac{D}{a}$, there exists $T>0$ such that $l:=h(T)>\frac{\pi}{2}\sqrt\frac{D}{a}$. This indicates that $a>\lambda_1$, where $\lambda_1$ is the first eigenvalue of the problem \eqref{lambda1}, and thus \eqref{stst} with $l$ replaced by $h$,
admits a unique positive solution $u_h(x)$ as long as $h>\hat h=:\frac{\pi}{2}\sqrt\frac{D}{a}$. Furthermore, it follows from  \cite{Crandall} that
$$u_h(x)=(h-\hat h)\cos \frac{\pi x}{2h}+o(h-\hat h)^2\varphi(x,\hat h),~~h-\hat h\ll1.
$$
for some $\varphi(x,\hat h)$. Define
$$F(h)=\mu\int_0^hu_h(x)w(h-x)dx.$$
Then, for $h$ close to $\hat h$, we have
\begin{equation}\nonumber
\begin{aligned}
F(h)&=\mu(h-\hat h)\int_0^h\cos \frac{\pi x}{2h}w(h-x)dx+o(h-\hat h)^2\mu\int_0^h\varphi(x,\hat h)w(h-x)dx\\&=\mu(h-\hat h)\left(\frac{\frac{c_1\pi}{2h}-c_1\alpha_1e^{-\alpha_1h}}{\alpha_1^2+(\frac{\pi}{2h})^2}-
\frac{\frac{c_2\pi}{2h}-c_2\alpha_2e^{-\alpha_2h}}{\alpha_2^2+(\frac{\pi}{2h})^2}\right)
+o(h-\hat h)^2.
\end{aligned}
\end{equation}
For $h$ close to $\hat{h}$, if $a$ is sufficiently large, it then follows from $c_1>c_2$ that
\begin{equation}\label{cond1}
\frac{c_1\sqrt{\frac{a}{D}}-c_1\alpha_1e^{-\frac{\pi}{2}\sqrt\frac{D}{a}\alpha_1}}{\alpha_1^2+\frac{a}{D}}
-\frac{c_2\sqrt{\frac{a}{D}}-c_2\alpha_2e^{-\frac{\pi}{2}\sqrt\frac{D}{a}\alpha_2}}{\alpha_2^2+\frac{a}{D}}>0,
\end{equation}
which implies that $F(h)>0$ for $h$ close to $\hat h$.

Let $\bar{h}>\frac{1}{\alpha_1-\alpha_2}\ln\frac{c_1}{c_2}=:c$. Given small $0<\epsilon<c$, from Theorem \ref{est} in Appendix, we can choose $a$ large enough such that, there exists a function $g_a(x)\geq0$ satisfying $g_a(x)/a\rightarrow0$ as $a\rightarrow\infty$, and
\begin{equation}\label{ubarh}
u_{\bar{h}}(x)= a-g_a(x),~~~~~~~~~x\in[0,\bar{h}-\epsilon].
\end{equation}
When $x\in[\bar{h}-\epsilon,\bar{h}]$, we have $u_{\bar{h}}(x)\leq a$, since $u\equiv a$ is an upper solution. Note that $w(\bar{h}-x)$ is positive for $x\in[\bar{h}-c,\bar{h}]$, and therefore is also positive for $x\in[\bar{h}-\epsilon,\bar{h}]$, see Figure \ref{uw}.
\begin{figure}[htbp]
\centering{\includegraphics[scale=0.55]{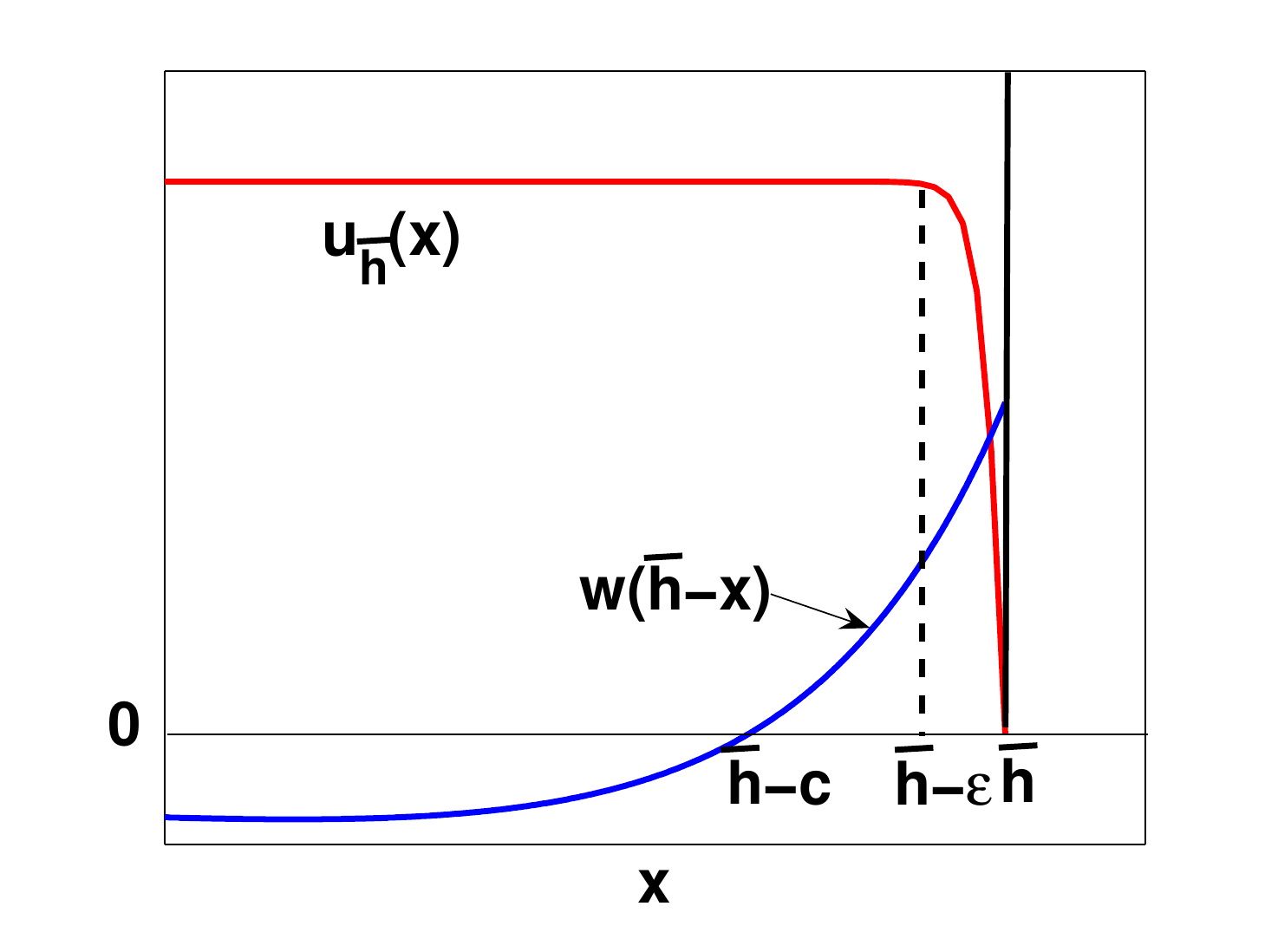}}
\caption{The graph of $u_{\bar{h}}(x)$ and $w(\bar{h}-x)$.}\label{uw}
\end{figure}
Thus,
$$
\begin{aligned}
F(\bar{h})&=\mu\int_0^{\bar{h}}u_{\bar{h}}(x)w(\bar{h}-x)dx=\mu\left(\int_0^{\bar{h}-\epsilon}+\int^{\bar{h}}_{\bar{h}-\epsilon}\right)
u_{\bar{h}}(x)w(\bar{h}-x)dx\\
&\leq\mu\int_0^{\bar{h}-\epsilon}[a-g_a(x)]w(\bar{h}-x)dx+
\mu\int^{\bar{h}}_{\bar{h}-\epsilon}aw(\bar{h}-x)dx=:T_1+T_2.\\
\end{aligned}
$$
For $T_1$, it follows from \eqref{ubarh} that,
$$
\begin{aligned}
T_1&=a\mu\int_0^{\bar{h}-\epsilon}w(\bar{h}-x)dx-\mu\int_0^{\bar{h}-\epsilon}g_a(x)w(\bar{h}-x)dx\\
&\leq a\mu\left[\frac{c_1}{\alpha_1}\left(e^{-\alpha_1\epsilon}-e^{-\alpha_1\bar{h}}\right)-
\frac{c_2}{\alpha_2}\left(e^{-\alpha_1\epsilon}-e^{-\alpha_2\bar{h}}\right)\right]\\
&\quad+o(a)\mu\left[\frac{c_1}{\alpha_1}\left(e^{-\alpha_1\epsilon}-e^{-\alpha_1\bar{h}}\right)+
\frac{c_2}{\alpha_2}\left(e^{-\alpha_1\epsilon}-e^{-\alpha_2\bar{h}}\right)\right],
\end{aligned}
$$
where $o(a)\geq0$ is a constant, such that $g_a(x)\leq o(a)$ and $o(a)/a\rightarrow0$ as $a\rightarrow \infty$.
On the other hand, for $T_2$, we have
$$
T_2=a\mu\left[\frac{c_1}{\alpha_1}\left(1-e^{-\alpha_1\epsilon}\right)-
\frac{c_2}{\alpha_2}\left(1-e^{-\alpha_2\epsilon}\right)\right].
$$
Note that, when $a$ is large enough, $\epsilon$ will be sufficiently small.
By \eqref{cond}, we can conclude that, if $a$ is large enough, then there exists a large $\bar{h}$ such that
$$
\begin{aligned}
F(\bar{h})&\leq a\mu\left[\frac{c_1}{\alpha_1}\left(1-e^{-\alpha_1\bar{h}}\right)-
\frac{c_2}{\alpha_2}\left(1-e^{-\alpha_2\bar{h}}\right)\right]\\
&\quad+o(a)\mu\left[\frac{c_1}{\alpha_1}\left(e^{-\alpha_1\epsilon}-e^{-\alpha_1\bar{h}}\right)+
\frac{c_2}{\alpha_2}\left(e^{-\alpha_1\epsilon}-e^{-\alpha_2\bar{h}}\right)\right]<0.
\end{aligned}
$$
Hence there exists an $h^*>\hat{h}$ such that $F(h^*)=0$. Substituting $h^*$ into \eqref{stst}, we can solve $u^*$.
\end{proof}

\begin{remark}
In the proof of Lemma \ref{main}, it is obvious that $\hat{h}<c<\bar{h}$, as long as $a$ is large enough and \eqref{cond} is satisfied. Since $w(h-x)>0$ for $\hat{h}<h<c$ and $0<x<h$, we know $F(h)>0$ for $\hat{h}<h<c$, and therefore the zero $h^*$ of $F(h)=0$ satisfies $c<h^*<\bar{h}$.
\end{remark}

\begin{remark}
Lemmas 2.4-2.6 show that the proposed model \eqref{model1} can have  trichotomous dynamics (vanishing, balancing and spreading), under proper conditions, which is quite different from the free boundary problem. Intuitively, this is mainly because the shrinking effect is considered in boundary equation of \eqref{model1}, so that a steady state will be attained if the shrinking and expansion effects are balanced. It also should be admitted that the conditions in Lemmas 2.4-2.6 are not easy to be testified, since the information of $h_\infty$ can not be obtained in advance.
\end{remark}

\section{Numerical Simulations}

In this section, we will numerically solve \eqref{model1}. The numerical scheme is to use the difference method. In detail, we use forward difference for time variable $t$ and central difference for spatial variable $x$, that is,
$$
\begin{aligned}
&u_t\approx\frac{u(t+\Delta t,x)-u(t,x)}{\Delta t},~h'\approx\frac{h(t+\Delta t)-h(t)}{\Delta t},\\
&u_{xx}\approx\frac{u(t,x+\Delta x)-2u(t,x)+u(t,x-\Delta x)}{\Delta x^2}.
\end{aligned}
$$
Therefore, if we know the solution $(u,h)$ at time $t_i$, then the solution at time $t_{i+1}$ is implicitly determined by
\begin{equation}\label{linsys}
\begin{aligned}
\frac{u(t_{i+1},x)-u(t_{i},x)}{\Delta t}&=\frac{u(t_{i+1},x+\Delta x)-2u(t_{i+1},x)+u(t_{i+1},x-\Delta x)}{\Delta x^2}+f(u(t_{i},x)),\\
\frac{h(t_{i+1})-h(t_{i})}{\Delta t}&=\mu\int_0^{h(t_i)}u(t_i,x)w(h(t_i)-x)dx.
\end{aligned}
\end{equation}
After re-discretizing the spatial variable $x$ with fixed $\Delta x$ at time $t_{i+1}$ (since the right boundary is changing), we can obtain the solution $(u,h)$ at time $t_{i+2}$ by \eqref{linsys}. Hence, this process allows us to obtain the numerical solution of \eqref{model1} for any time $t>0$.

\subsection{With Intrinsic Growth}

To perform simulations for the free boundary problem \eqref{model1}, we fix
\begin{equation}\label{par1}
D=1,~\mu=1,~c_2=1,~\alpha_1=1.9,~\alpha_2=1,~f(u)=ru(a-u),~r=1,~a=5.
\end{equation}
Next, we shall examine the effect of the parameters of the kernel function on the dynamics of \eqref{model1}. Set $c_1=1.5$, $h_0=3$.  Then, it is numerically found that the solution $(u(t,x),h(t))$ of \eqref{model1} satisfies
\begin{equation}\label{vanishing}
\lim_{t\rightarrow\infty} u(t,x)=0,~~x\in(0,h(t)),~~~~~~\lim_{t\rightarrow\infty} h(t)=h_*(\approx0.7613),
\end{equation}
meaning that the population will go extinct eventually, see Figure \ref{ex11}. Now, we choose $c_1$ as the varying parameter, and expect that the increment of $c_1$ will help the population spread, since a higher value of $c_1$ indicates the animals more likely move towards the boundary, leading to the expansion of the boundary. To verify this, $c_1$ is increased to a critical value $c_1^*\approx1.6$. Then, we find
\begin{equation}\label{balancing}
\lim_{t\rightarrow0} u(t,x)=\bar{u}(x),~~x\in(0,h(t)),~~~~~~\lim_{t\rightarrow\infty} h(t)=h^*(\approx0.9216),
\end{equation}
where $\bar{u}(x)$ is the nonconstant steady state of
\begin{equation}\label{fix}
\begin{aligned}
&u_t=Du_{xx}+f(u),~~~t>0,~0<x<h^*,\\
&u_x(0)=0,~~~~u(h^*)=0,
\end{aligned}
\end{equation}
that is, the species will be established in a fixed range boundary $(0,h^*)$.
The system \eqref{fix} can be regarded as the associated reaction-diffusion equation of \eqref{model1} with fixed boundary.
As we keep raising the value of $c_1$ up to another critical value $c_1^{**}\approx2.8$, it is observed that the dynamics do not change qualitatively for $c_1\in(c_1^*,c_1^{**})$, except that $(u(x,t),h(t))$ will tend to different steady states depending on $c_1$, see Figure \ref{ex11} (b) and (c). When the parameter $c_1$ exceeds $c_1^{**}$, we find that
\begin{equation}\label{spreading}
\lim_{t\rightarrow\infty} u(t,x)={a},~~~~\lim_{t\rightarrow\infty} h(t)=\infty
\end{equation}
for $|x-h(t)|>\delta$ with $\delta>0$. Moreover,
the right boundary $h(t)$ is strictly increasing with a constant spreading speed, denoted by $\rho$, see Figure \ref{ex11} (d). This verifies Lemma 2.5.
Recall from \cite{Du2010} that, the spreading-vanishing dichotomy holds for model \eqref{DuLin}. When the boundary equation becomes nonlocal, \eqref{model1} will exhibit new dynamics (balancing), that does not show up in \eqref{DuLin}. In this case, the dynamics of \eqref{model1} for large time is similar to that of \eqref{fix}.
From Figure \ref{ex12} (a), we can observe that, in the case of spreading, the speed is an increasing function of $c_1$ with all the other parameters being fixed, and the value of $\rho$ almost linearly depends on $c_1$.

Similarly, we can also detect the impact of the other parameters $c_1$, $\alpha_1$ and $\alpha_2$ in the weight function on the dynamics of \eqref{model1}. The phenomena of vanishing, balancing and spreading could be observed by altering the values of each of these three parameters. In the case of spreading, the speed $\rho$ is an increasing function of $\alpha_2$, but is decreasing with respect to $c_2$ and $\alpha_1$, respectively, see Figure \ref{ex12} (b)-(d).

When the phenomenon of spreading happens for \eqref{model1}, we can also examine whether the initial range boundary size $h_0$ and initial population $u_0(x)$ will influence the spreading speed. It turns out that the speed $\rho$ is independent of $h_0$ and $u_0(x)$. However, if $h_0$ is declined to a critical value (without changing $u_0(x)$), say $\hat{h}$, then the population fails to spread and may tend to either a nonconstant steady state or zero (that is, balancing or vanishing). In this situation, when we increase the initial total population, then the spreading for \eqref{model1} will be observed again. This means, for \eqref{model1}, there exists a critical range boundary size $\hat{h}$, such that whether spreading or balancing (or vanishing) occurs relies on the initial population density $u_0$, for $h_0<\hat{h}$. If we choose $c_1=2.9$ and $u_0(x)=0.3(h_0^2-x^2)$, then $\hat{h}\approx0.682$. Unlike the impact of the parameters in the weight function $w$, it turns out that \eqref{spreading} holds for $h>\hat{h}$ and \eqref{vanishing} holds for $h<\hat{h}$, that is, the variation of $h_0$ can only lead to spreading-vanishing dichotomy for \eqref{model1}, see Figure \ref{ex13}. Let $h_0=0.67<\hat{h}$. Then, we find that the solution $u(x,t)$ of \eqref{model1} with the initial population distribution $u_0(x)=0.3(h_0^2-x^2)$ approaches zero, as $t\rightarrow\infty$. As we increase the initial population density to $u_0(x)=0.6(h_0^2-x^2)$, it is found out that the population will spread again, see Figure \ref{ex14}.

\begin{figure}[htbp]
\begin{minipage}{0.5\linewidth}
\centering\subfigure[$c_1=1.5$]
{\includegraphics[width=2in]
{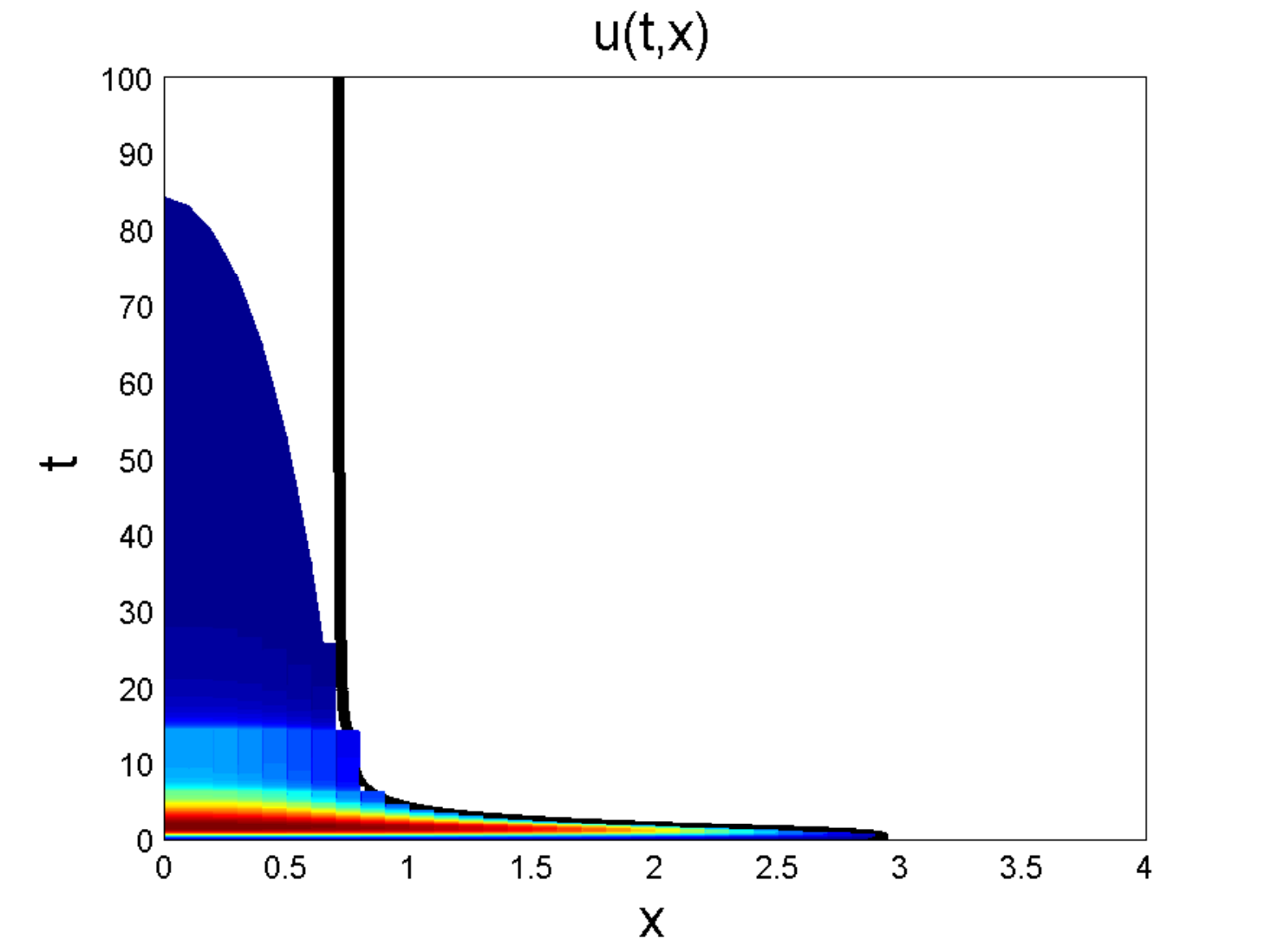}}
\end{minipage}	
\begin{minipage}{0.5\linewidth}
\centering\subfigure[$c_1=1.6$]
{\includegraphics[width=2in]
{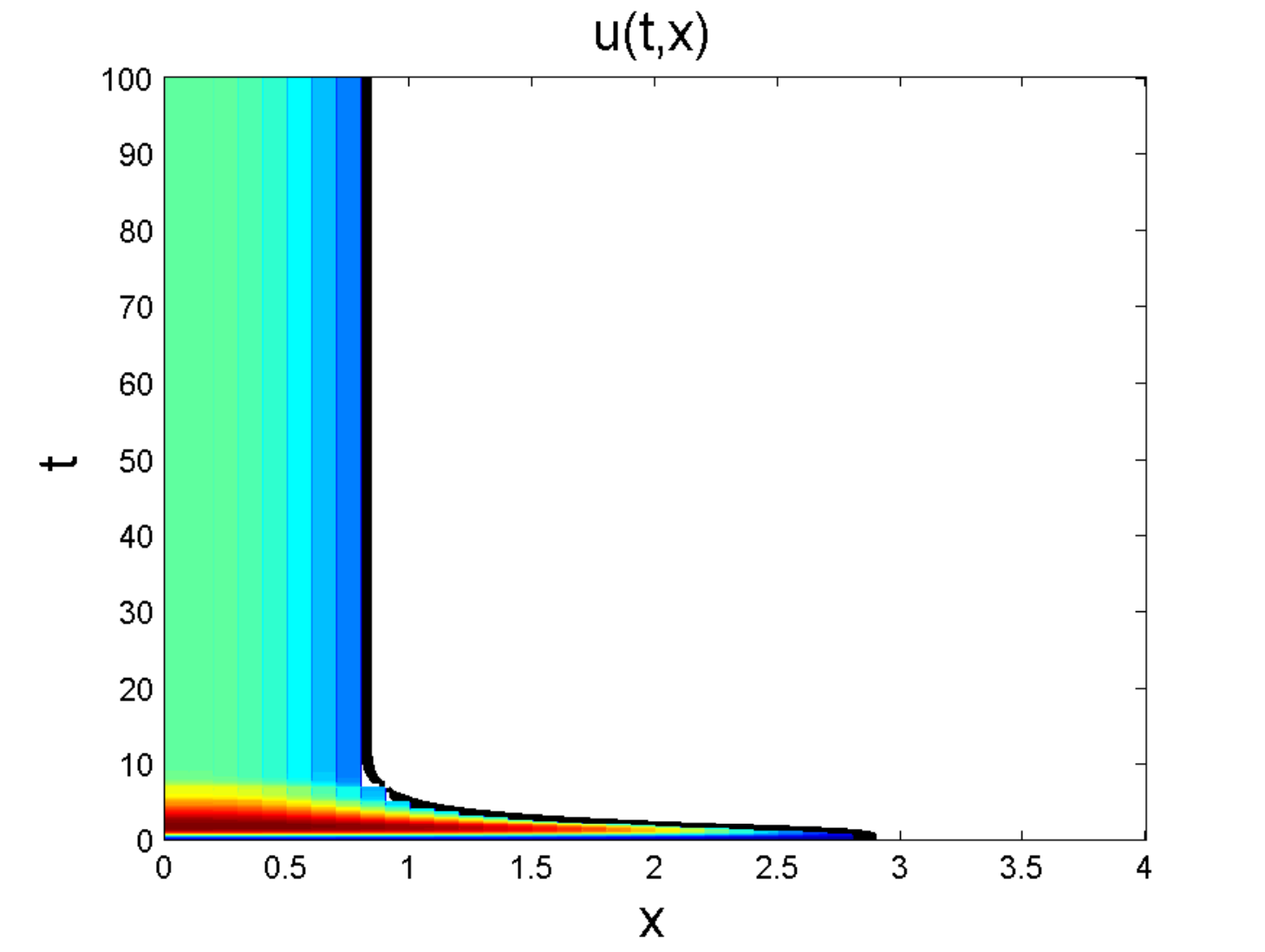}}
\end{minipage}
\begin{minipage}{0.5\linewidth}
\centering\subfigure[$c_1=2.8$]
{\includegraphics[width=2in]
{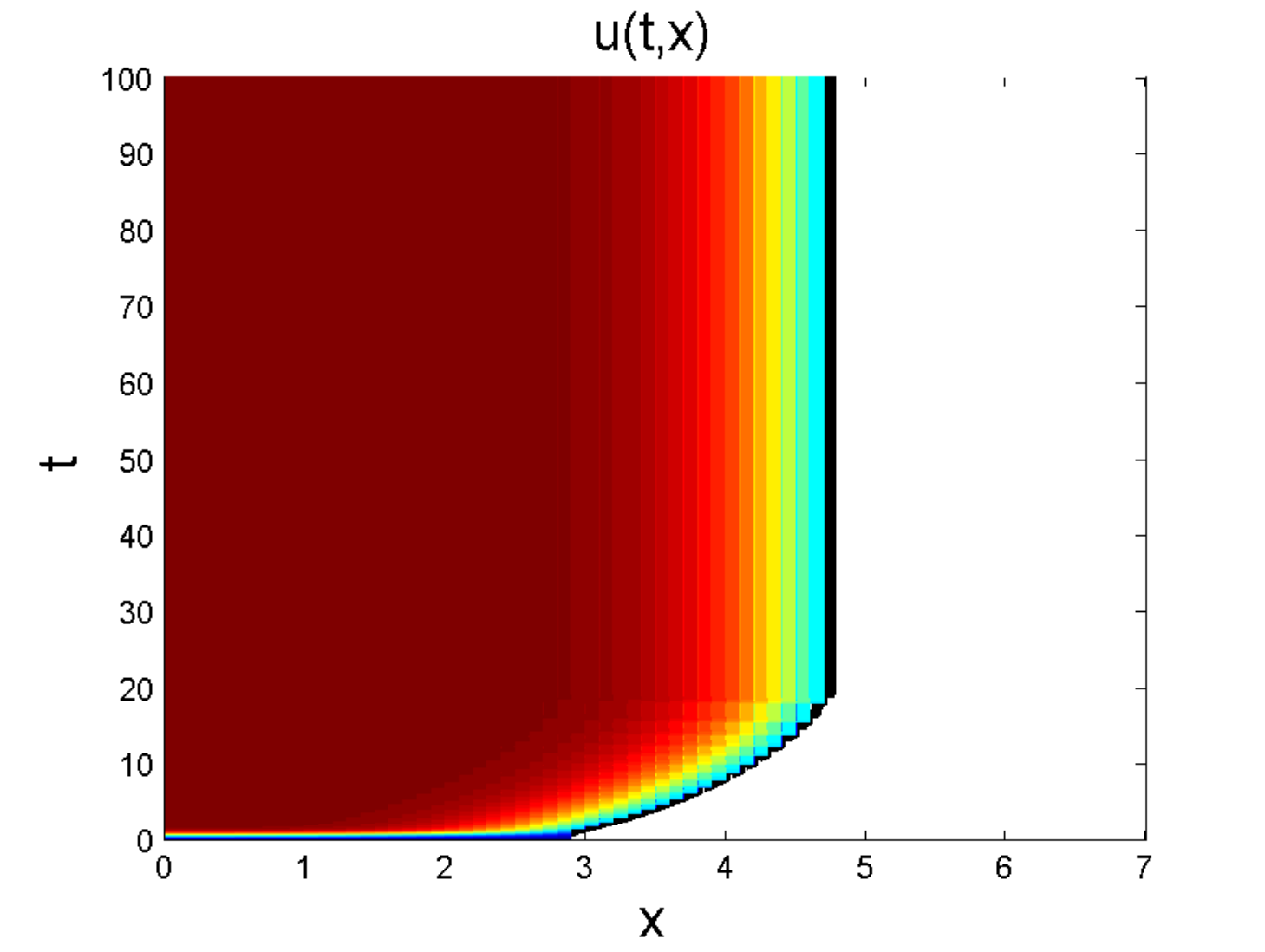}}
\end{minipage}
\begin{minipage}{0.5\linewidth}
\centering\subfigure[$c_1=2.9$]
{\includegraphics[width=2in]
{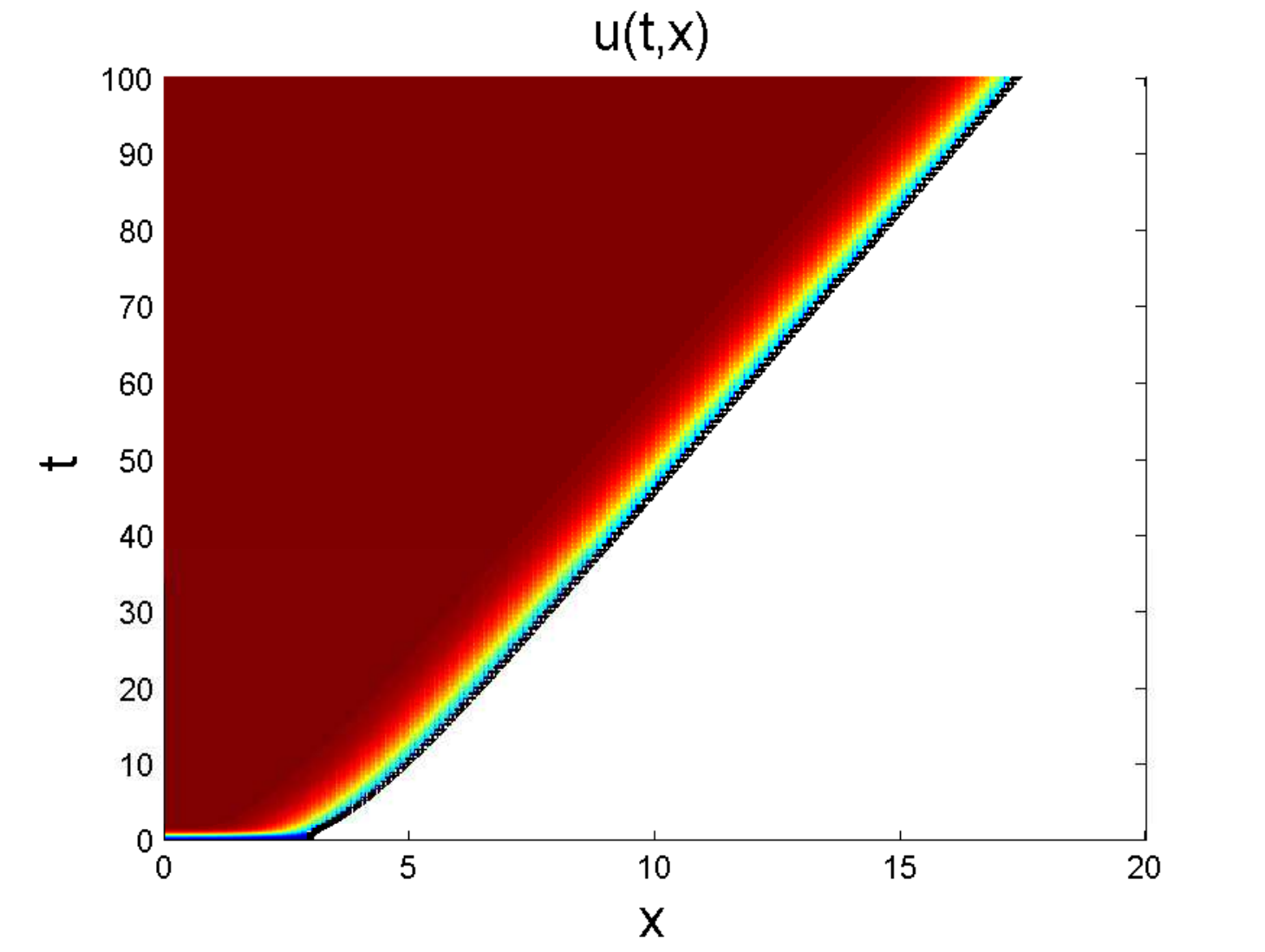}}
\end{minipage}
	\caption{The solutions of \eqref{model1} for different choices of $c_1$, with the initial function $u_0(x)=0.01(h_0^2-x^2)$. All the other parameter values are given in \eqref{par1}. From $(a)$ to $(d)$, the dynamics of \eqref{model1} varies from vanishing for relatively small $c_1=1.5$, to balancing for intermediate values of $c_1\in[1.6,2.8]$, and then to spreading for large $c_1=2.9$. The color represents population density. The red means the highest population density, while the blue means lower density.}\label{ex11}
\end{figure}
\subsection{Without Intrinsic Growth and Spreading Speed}

Now, we examine the nonlocal effect of the population on the boundary, without considering the intrinsic growth of species (that is, $r=0$). Let $c_1=3.6$. It is observed that the boundary evolves in different manners for different choices of $h_0$. However, the phenomenon of range expansion is absent for \eqref{model1} with $r=0$, see the black curves in Figure \ref{ex10}. Specifically, there is a critical value $\bar{h}$ such that $\displaystyle \lim_{t\rightarrow\infty} h(t) < h_0$ ($>h_0$ resp.) provided that $h_0>\bar{h}$ ($h_0<\bar{h}$ resp.). In Figure \ref{ex10} (a) or (b), the boundary $h(t)$ approaches a constant, as $h_0<\bar h$, while $h(t)$ in Figure \ref{ex10} (c) or (d) decreases, as long as $h_0$ exceeds $\bar h$. Therefore, $\bar h$ can be viewed as a threshold for \eqref{model1} changing its range from shrinkage to balancing. In addition, for each fixed $h_0$, we can see that the population growth will benefit its spatial expansion, since the value of $h(t)$ becomes larger as $r$ increases. For instance, in Figure \ref{ex10} (d), for $h_0=20$, if $r=0$, the phenomenon of range shrinkage is observed. As we increase $r$, either range balancing or expansion can happen, depending on the values of $r$, see red and blue curves in Figure \ref{ex10} (d).

\begin{figure}[h]
\begin{minipage}{0.5\linewidth}
\centering\subfigure[$\rho$ v.s. $c_1$]
{\includegraphics[width=2in]
{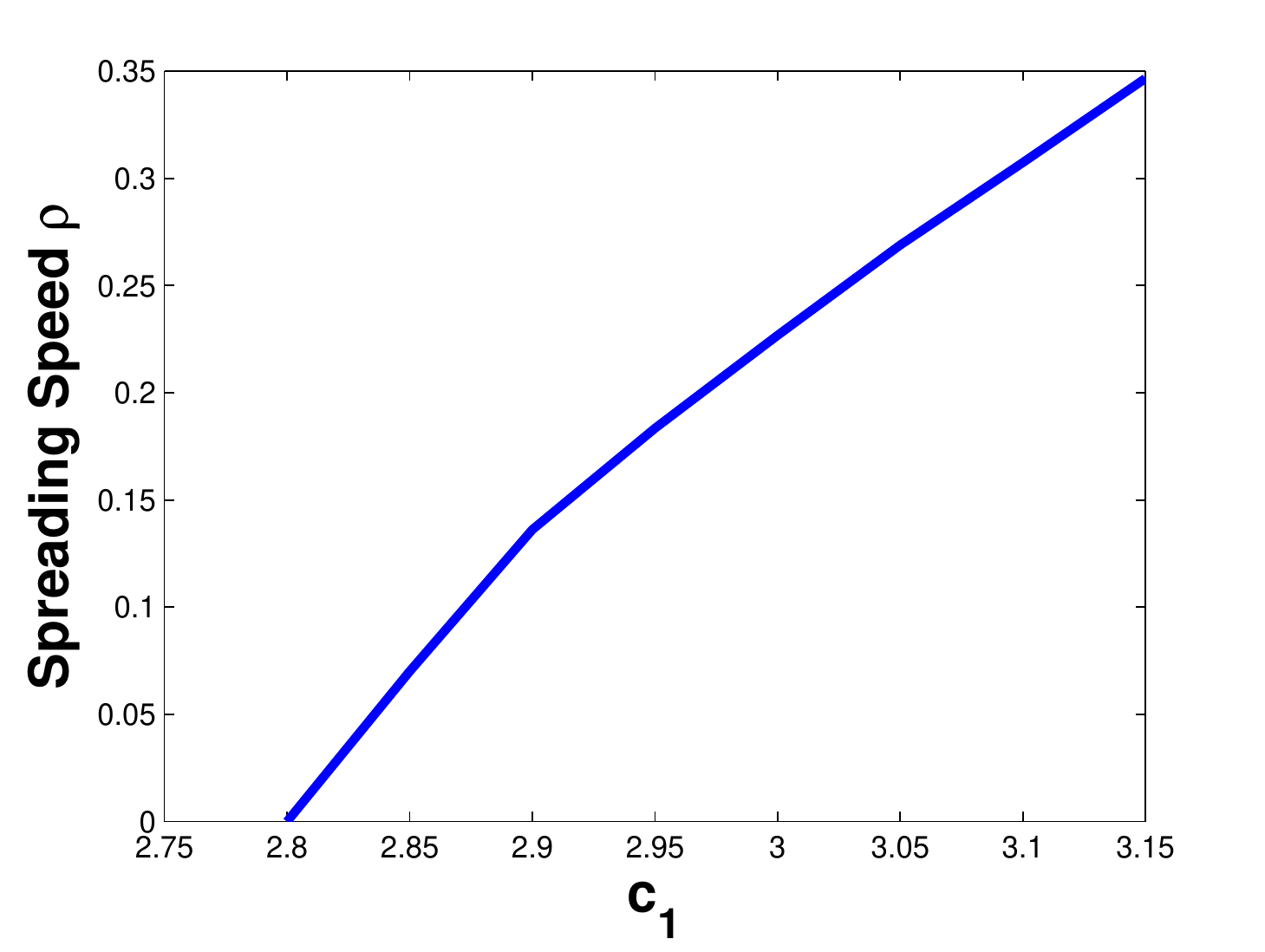}}
\end{minipage}	
\begin{minipage}{0.5\linewidth}
\centering\subfigure[$\rho$ v.s. $c_2$]
{\includegraphics[width=2in]
{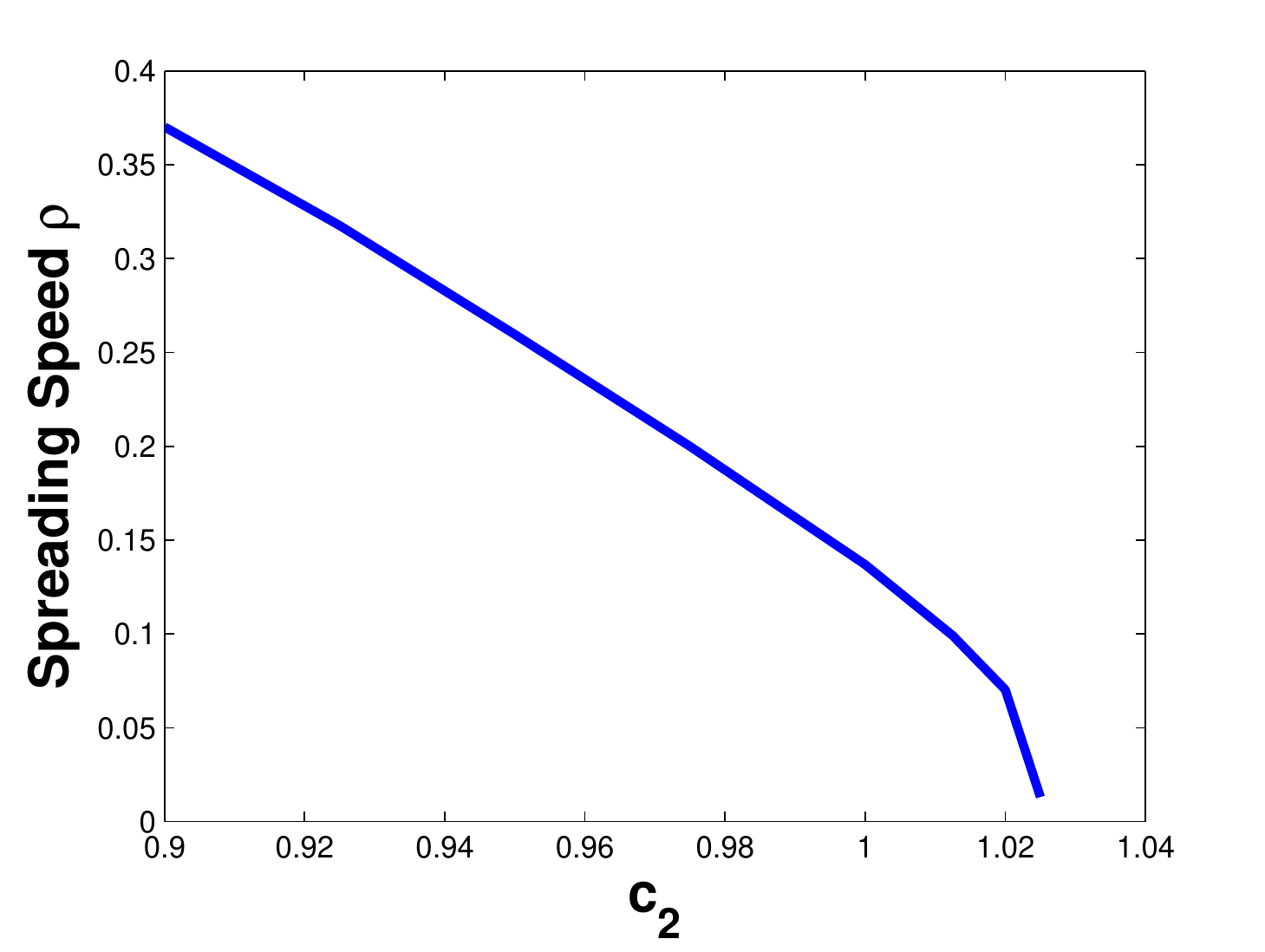}}
\end{minipage}
\begin{minipage}{0.5\linewidth}
\centering\subfigure[$\rho$ v.s. $\alpha_1$]
{\includegraphics[width=2in]
{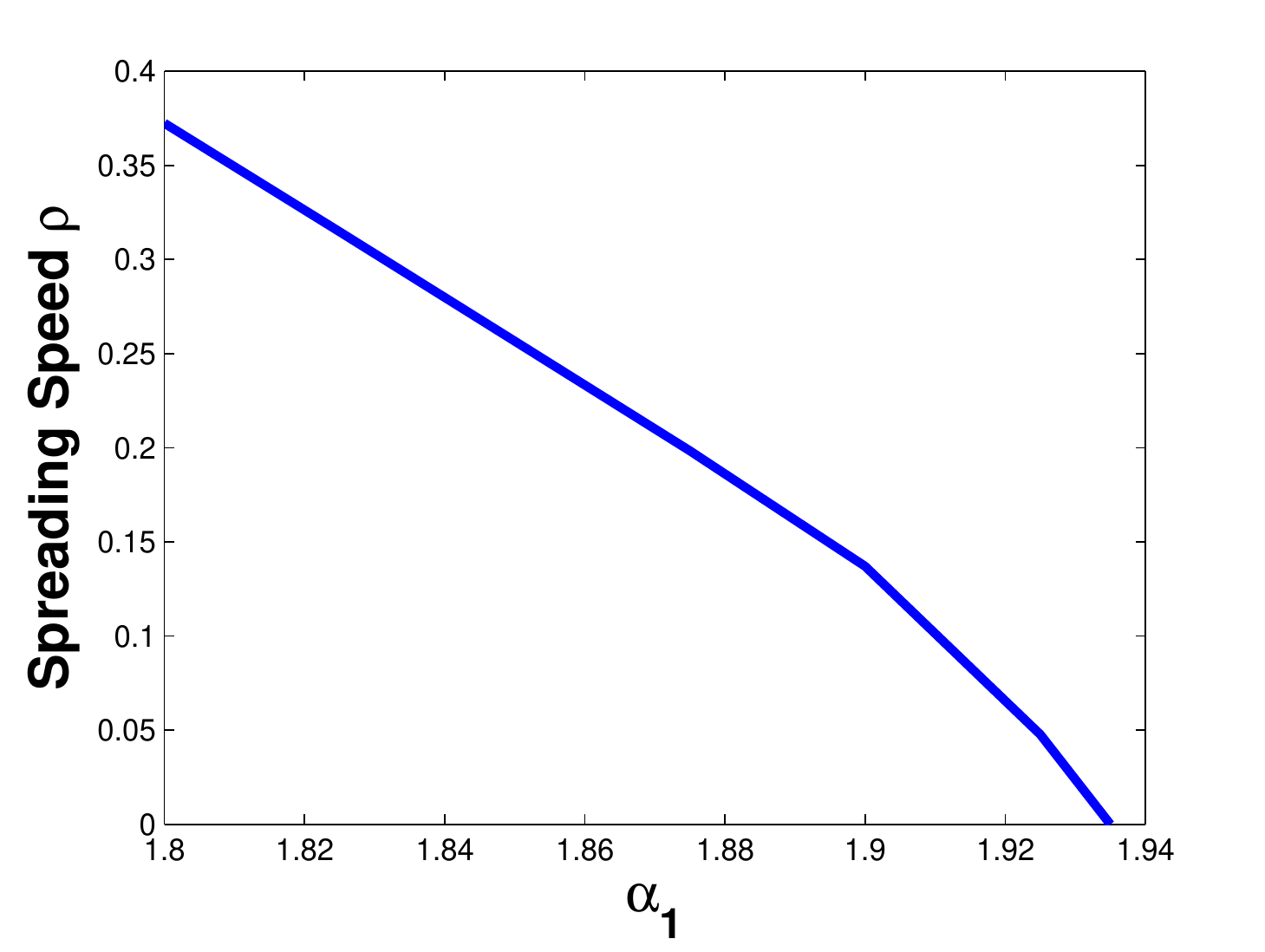}}
\end{minipage}
\begin{minipage}{0.5\linewidth}
\centering\subfigure[$\rho$ v.s. $\alpha_2$]
{\includegraphics[width=2in]
{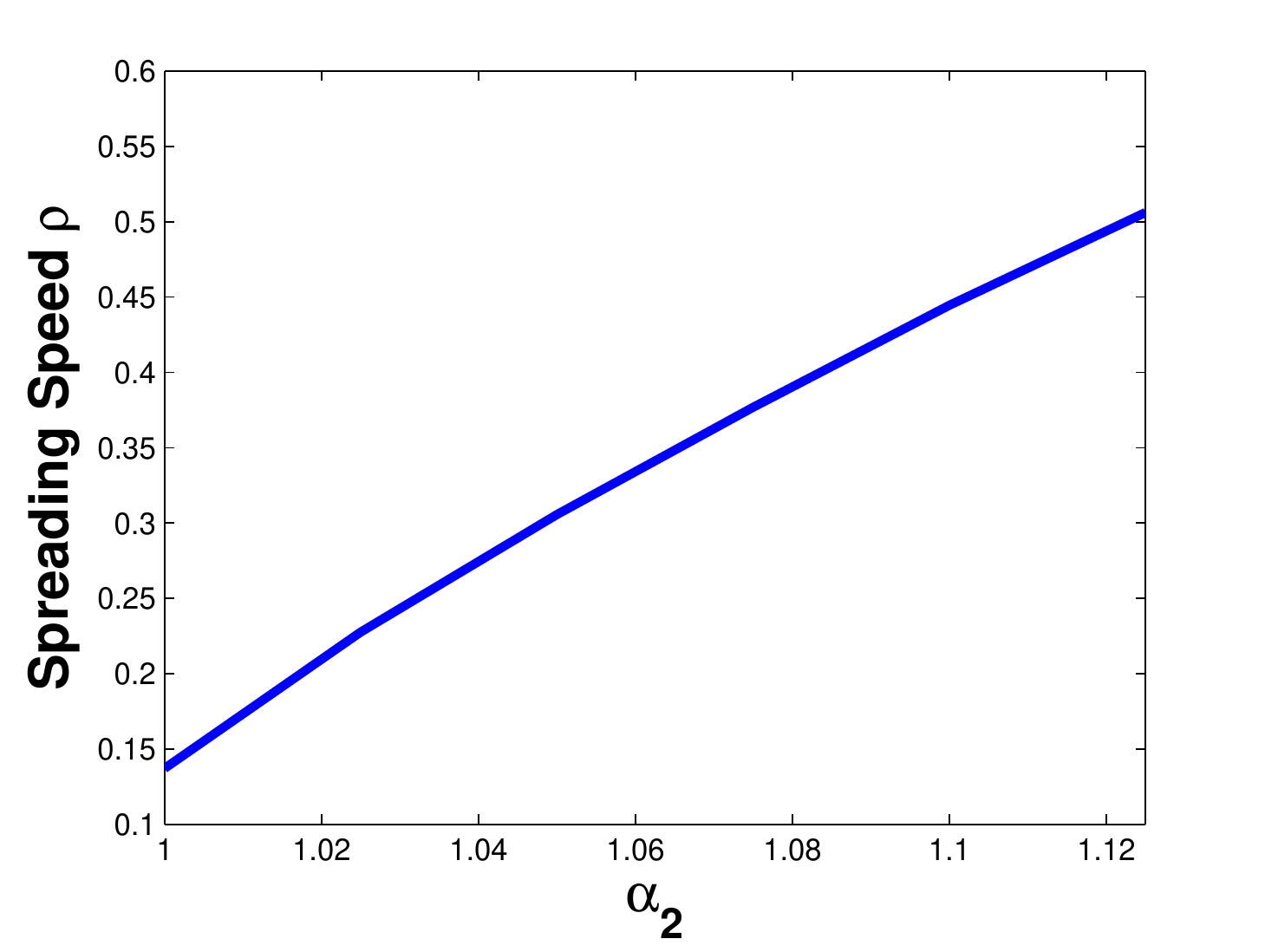}}
\end{minipage}
	\caption{The impact of parameters $c_1,c_2,\alpha_1,\alpha_2$ on the spreading speed $\rho$.}\label{ex12}
\end{figure}

\begin{figure}[h]
\begin{minipage}{0.5\linewidth}
\centering\subfigure[$h_0=0.67.$]
{\includegraphics[width=2in]
{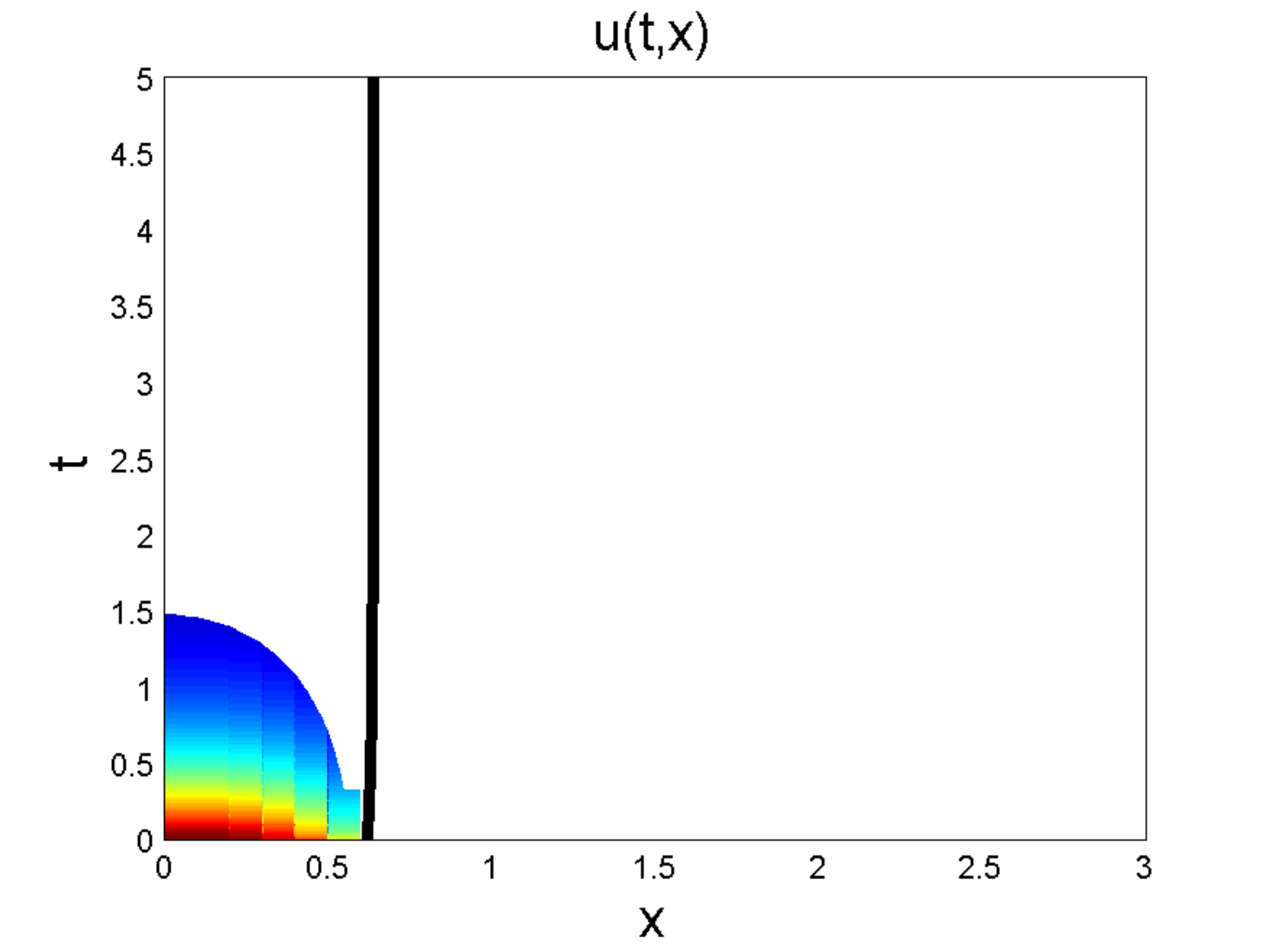}}
\end{minipage}	
\begin{minipage}{0.5\linewidth}
\centering\subfigure[$h_0=0.69.$]
{\includegraphics[width=2in]
{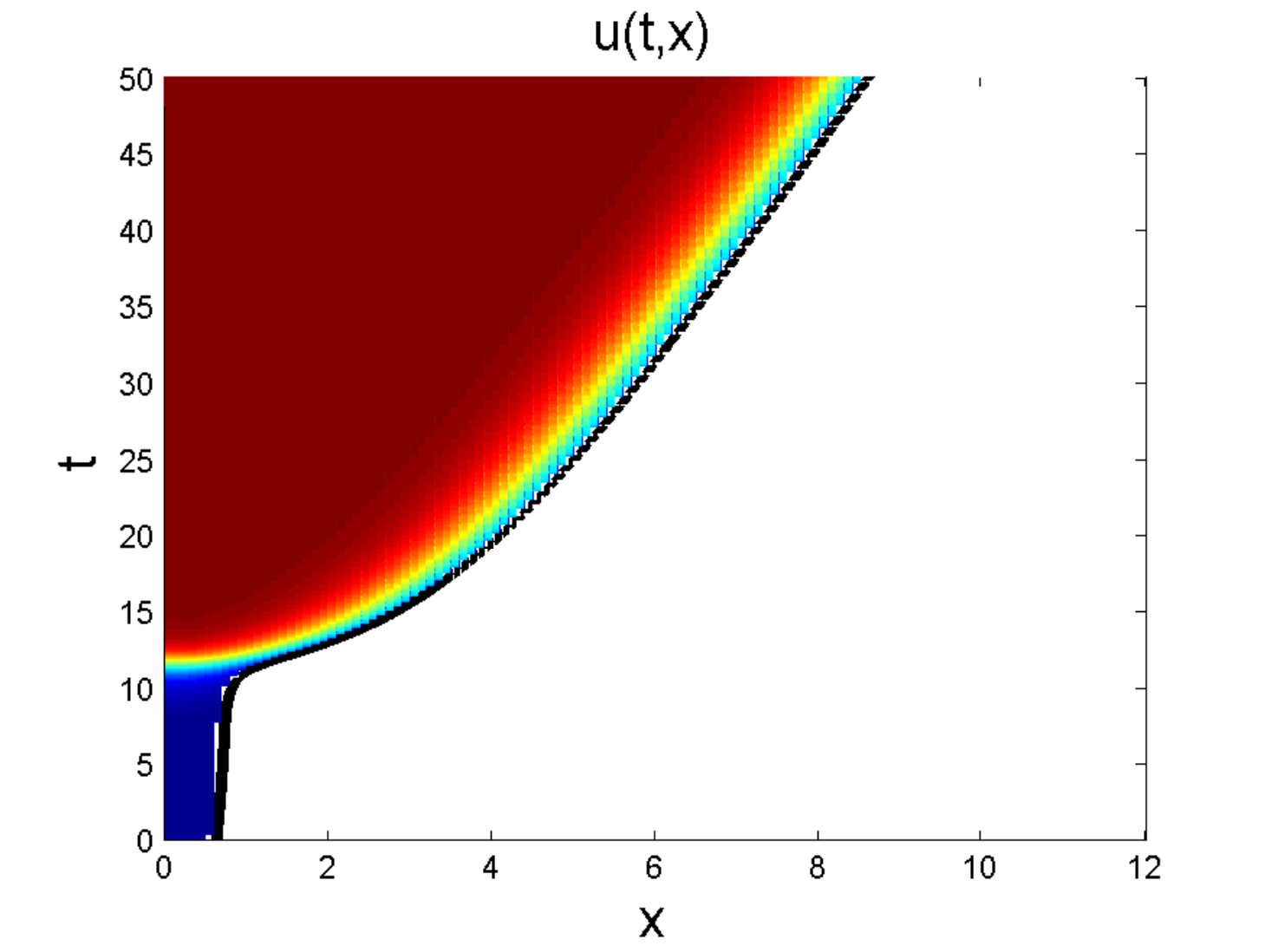}}
\end{minipage}
\caption{Spreading-vanishing dichotomy for \eqref{model1}, with different choices of $h_0$. Here, $c_1=2.9$, $u_0(x)=0.3(h_0^2-x^2)$ and all the other parameters are given by \eqref{par1}.}\label{ex13}
\end{figure}

\begin{figure}[h]
\begin{minipage}{0.5\linewidth}
\centering\subfigure[$u_0(x)=0.3(h_0^2-x^2).$]
{\includegraphics[width=2in]
{figs/h067.eps}}
\end{minipage}	
\begin{minipage}{0.5\linewidth}
\centering\subfigure[$u_0(x)=0.6(h_0^2-x^2).$]
{\includegraphics[width=2in]
{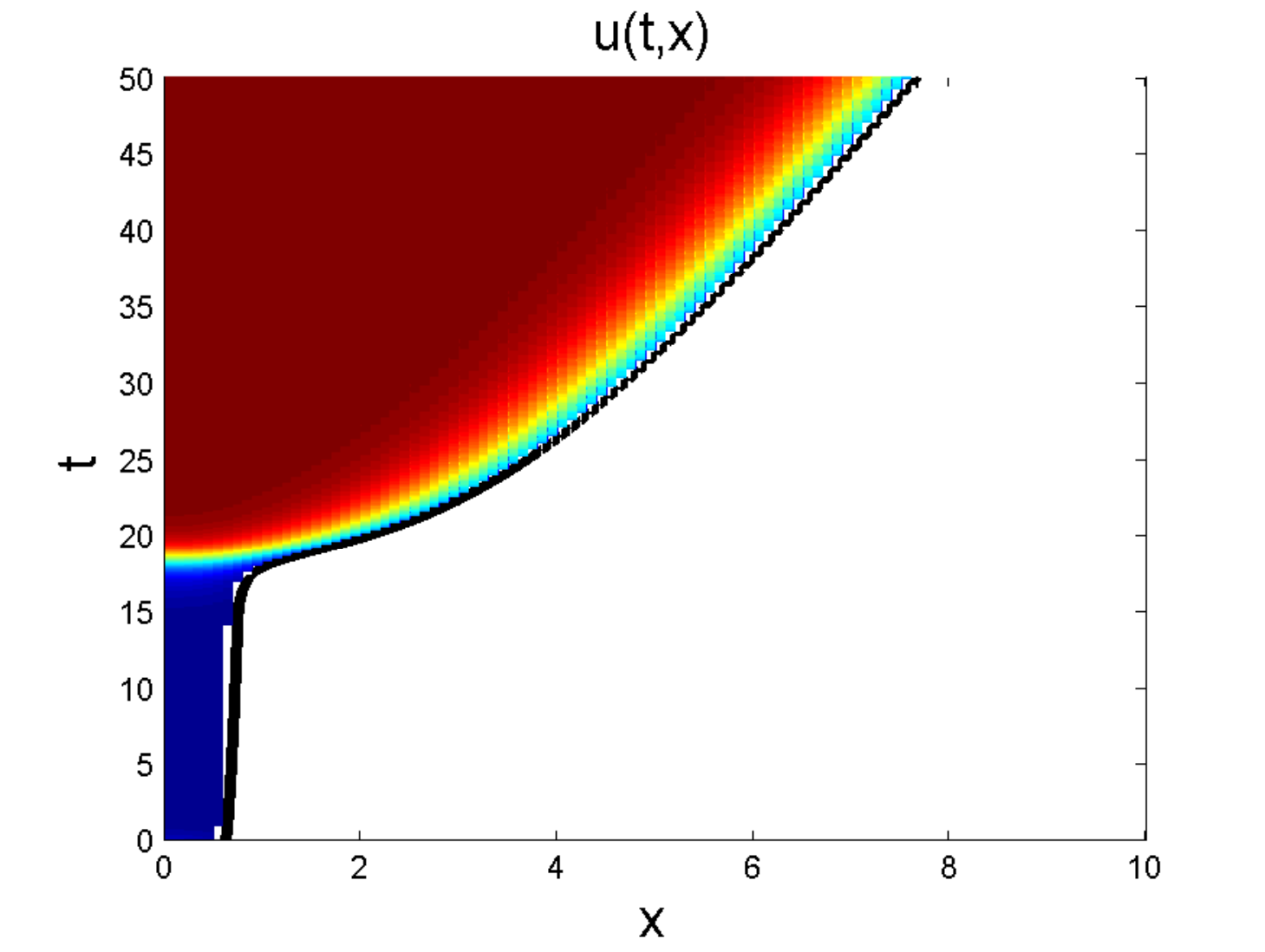}}
\end{minipage}
\caption{Spreading-vanishing dichotomy for \eqref{model1}, with different choices of $u_0(x)$. Here, $c_1=2.9$, $h_0=0.67$ and all the other parameters are given by \eqref{par1}.}\label{ex14}
\end{figure}

\begin{figure}[h]
\begin{minipage}{0.5\linewidth}
\centering\subfigure[$h_0=5.$]
{\includegraphics[width=2in]
{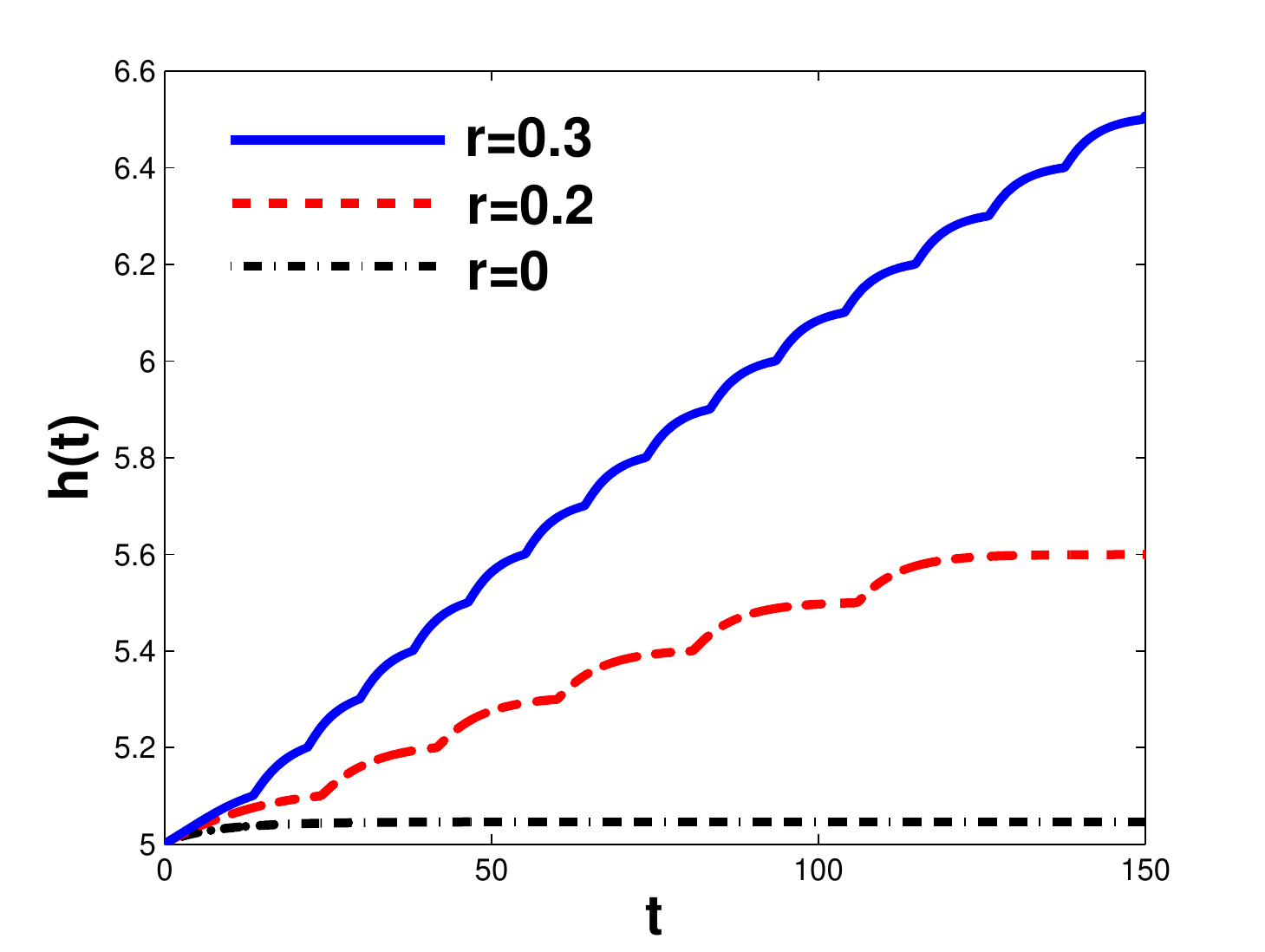}}
\end{minipage}	
\begin{minipage}{0.5\linewidth}
\centering\subfigure[$h_0=12.$]
{\includegraphics[width=2in]
{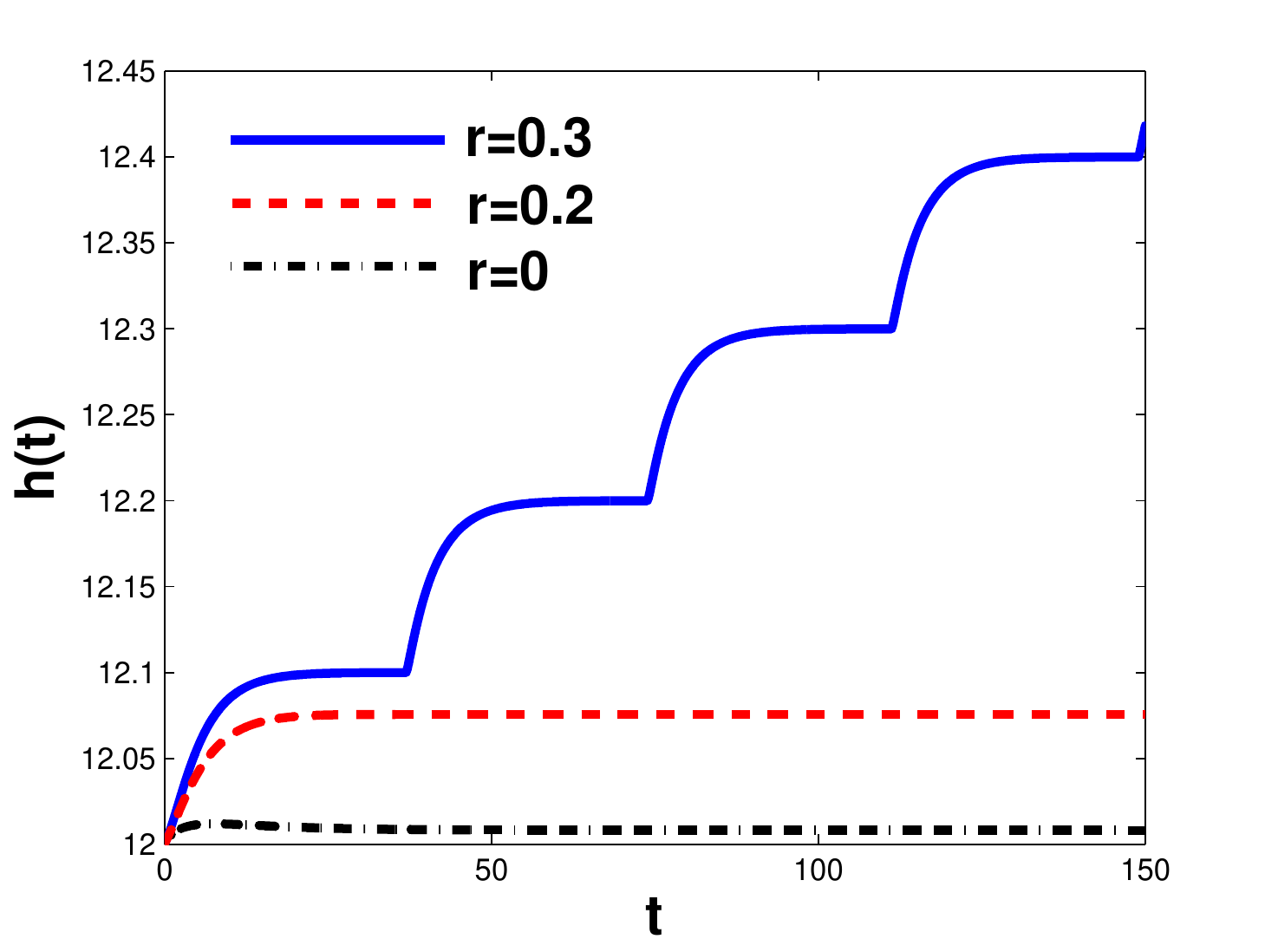}}
\end{minipage}
\begin{minipage}{0.5\linewidth}
\centering\subfigure[$h_0=19.$]
{\includegraphics[width=2in]
{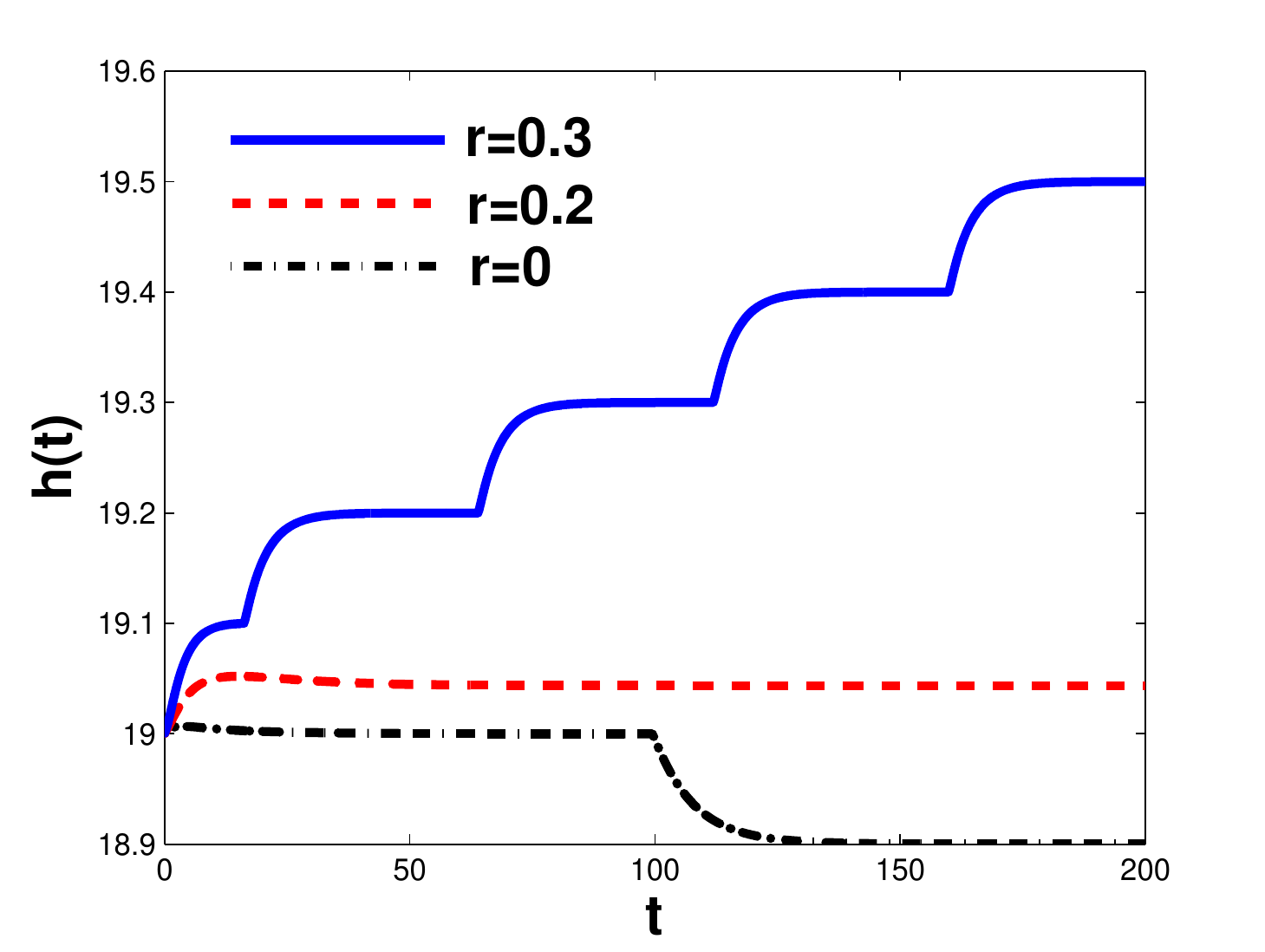}}
\end{minipage}
\begin{minipage}{0.5\linewidth}
\centering\subfigure[$h_0=20.$]
{\includegraphics[width=2in]
{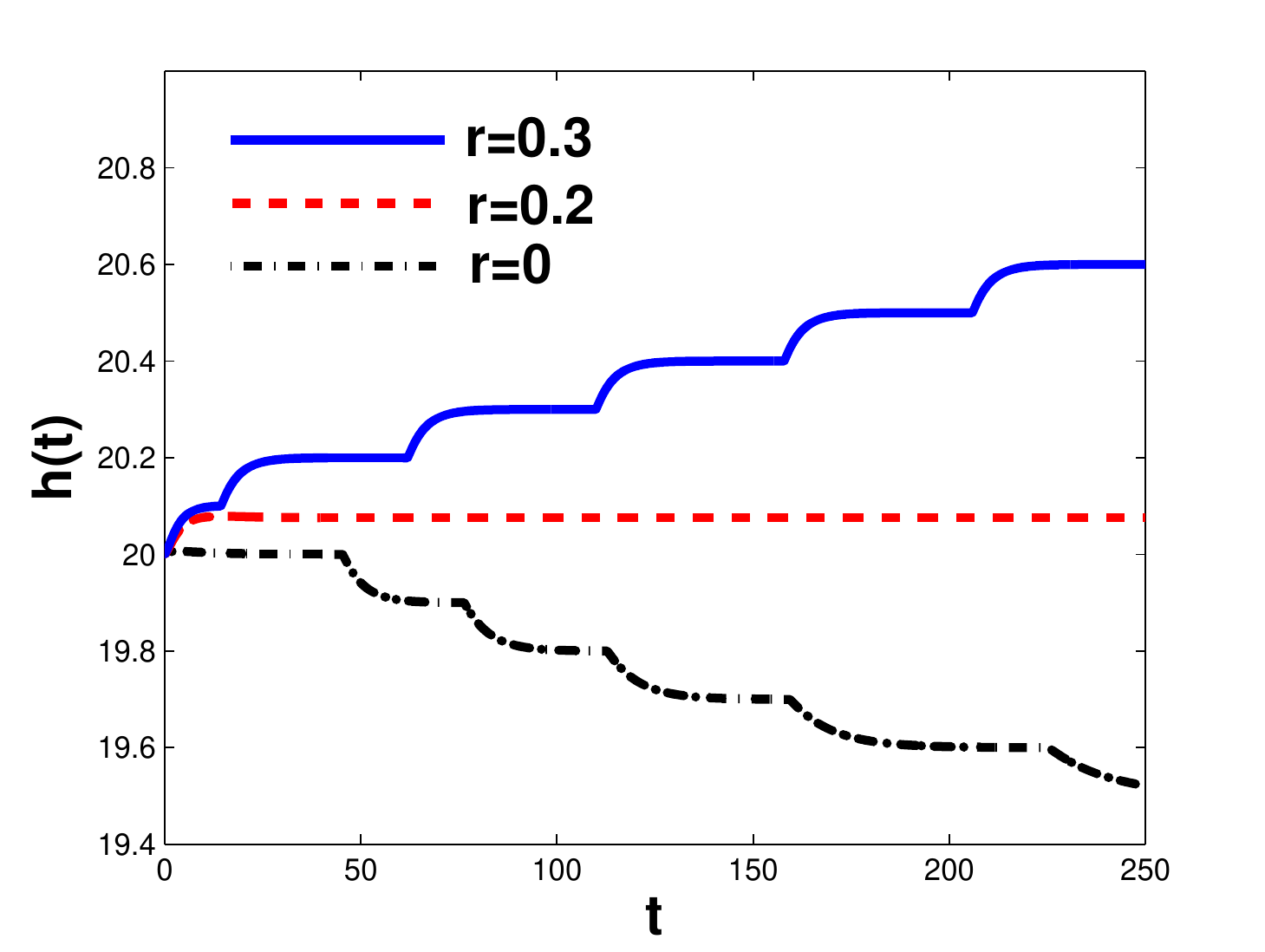}}
\end{minipage}
	\caption{The boundary change of \eqref{model1} for different choices of $h_0$. Here, $c_1=3.6$, $u_0(x)=0.01(h_0^2-x^2)$ and all the other parameters are given in \eqref{par1}.}\label{ex10}
\end{figure}

Finally, we shall compare the spreading speed $\rho$  with the one for the model \eqref{DuLin} (where the free boundary is described by Stefan condition).
It has been proved in Proposition 4.1 in \cite{Du2010} that the spreading speed $k_0$ for \eqref{DuLin} is uniquely determined by
$$
\mu U_{k_0}'(0)=k_0,
$$
where $U_k$ is the positive solution of
\begin{equation}\label{2ode}
-DU''+k U'=rU(1-U)
\end{equation}
with $U(0)=0$. It follows from \cite{2012K} that \eqref{2ode} have a unique positive solution for $0\leq k<2\sqrt{rD}$. Therefore, the spreading speed $k_0$ of \eqref{DuLin} is less than $2\sqrt{rD}$. For \eqref{model1}, it can be seen from Figure \ref{ex12} that, the spreading speed is most likely less than $2$ (note that $r=1$ and $D=1$ in simulations), no matter how we vary a single parameter $c_1$, $c_2$, $\alpha_1$ or $\alpha_2$. Recall that Stefan condition is a special case of \eqref{nonlocalh}, by a proper choice of weigh function. Thus, one may expect that the spreading speed of \eqref{model1} will not exceed the critical value $2\sqrt{rD}$. This suggest that the nonlocal effect on the boundary equation slows down the spreading speed of $u$, comparing with the classic Fisher-KPP equation on unbounded domain.

\section{Double Front Spreading}\label{twoside}

In this section, we will consider \eqref{DuLin} with two free boundaries, that is,
\begin{equation}\label{model2}
\begin{aligned}
&u_t=Du_{xx}+f(u),~~~~~t>0,~g(t)<x<h(t),\\
&g'(t)=\mu\int_{h(t)}^{g(t)}u(t,x)G(x-g(t))dx,t>0,\\
&h'(t)=\mu\int_{g(t)}^{h(t)}u(t,x)H(h(t)-x)dx,t>0,\\
&u(t,g(t))=u(t,h(t))=0,~~~~~~t>0,\\
&-g(0)=h(0)=h_0,~~~u(0,x)=u_0(x),~~-h_0\leq x\leq h_0,
\end{aligned}
\end{equation}
where
$$
\begin{aligned}
G(x-g(t))&=c_1e^{-\alpha_1(x-g(t))}-c_2e^{-\alpha_2(x-g(t))},\\
H(h(t)-x)&=c_3e^{-\alpha_3(h(t)-x)}-c_4e^{-\alpha_4(h(t)-x)},\\
\end{aligned}
$$
with $c_1>c_2>0$, $\alpha_1>\alpha_2>0$, $c_3>c_4>0$ and $\alpha_3>\alpha_4>0$.
Assume that $(H1)$ and $(H2)$ hold, using a similar argument as in Theorem \ref{eu}, one can show the local existence and uniqueness of solutions to \eqref{model2}. Here, we only carry out the numerical simulations to examine the impact of the parameters in the model on the dynamics of \eqref{model2}, especially on the spreading speed.

Suppose $c_1=c_3$, $c_2=c_4$, $\alpha_1=\alpha_3$ and $\alpha_2=\alpha_4$. Then, three possible dynamics (including vanishing, balancing and spreading) of \eqref{model2} can take place, see Figure \ref{ex21}. Moreover, in the case of spreading, the speed $\rho$ of boundary expansion on both sides are the same, and it is an increasing (decreasing) function of $c_1$ and $\alpha_2$ ($c_2$ and $\alpha_1$). Again, there is a critical value of $h_0$, also denoted by $\hat{h}$, such that the initial population density $u_0(x)$ will determine the occurrence of range expansion, when $h<\hat{h}$. These numerical results are not presented here, since they are about the same as those for \eqref{model1} except that the two boundaries move in a symmetric way (Figure \ref{ex21}).

\begin{figure}[h]
\begin{minipage}{0.5\linewidth}
\centering\subfigure[$c_1=2.$]
{\includegraphics[width=2in]
{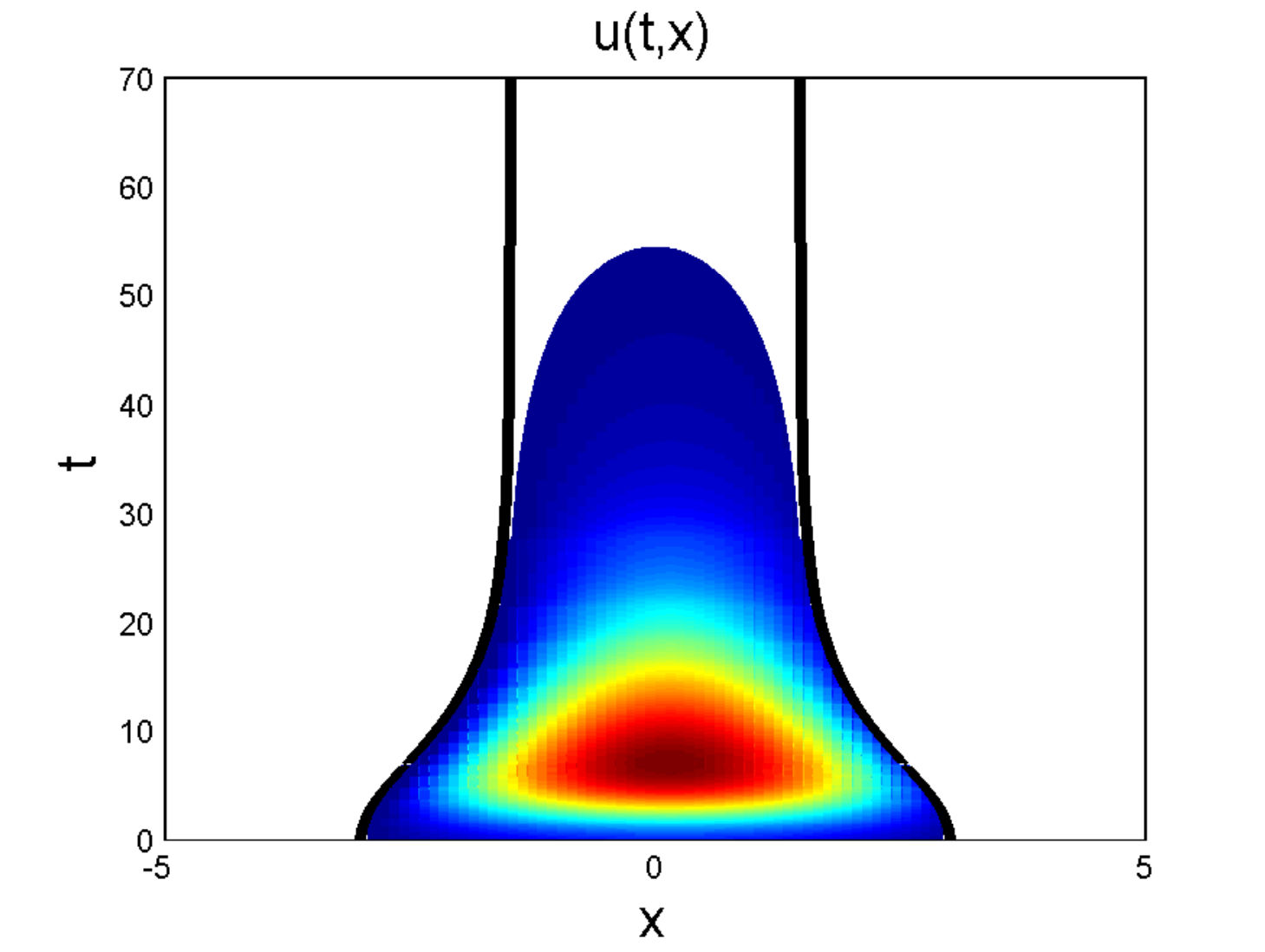}}
\end{minipage}	
\begin{minipage}{0.5\linewidth}
\centering\subfigure[$c_1=2.4.$]
{\includegraphics[width=2in]
{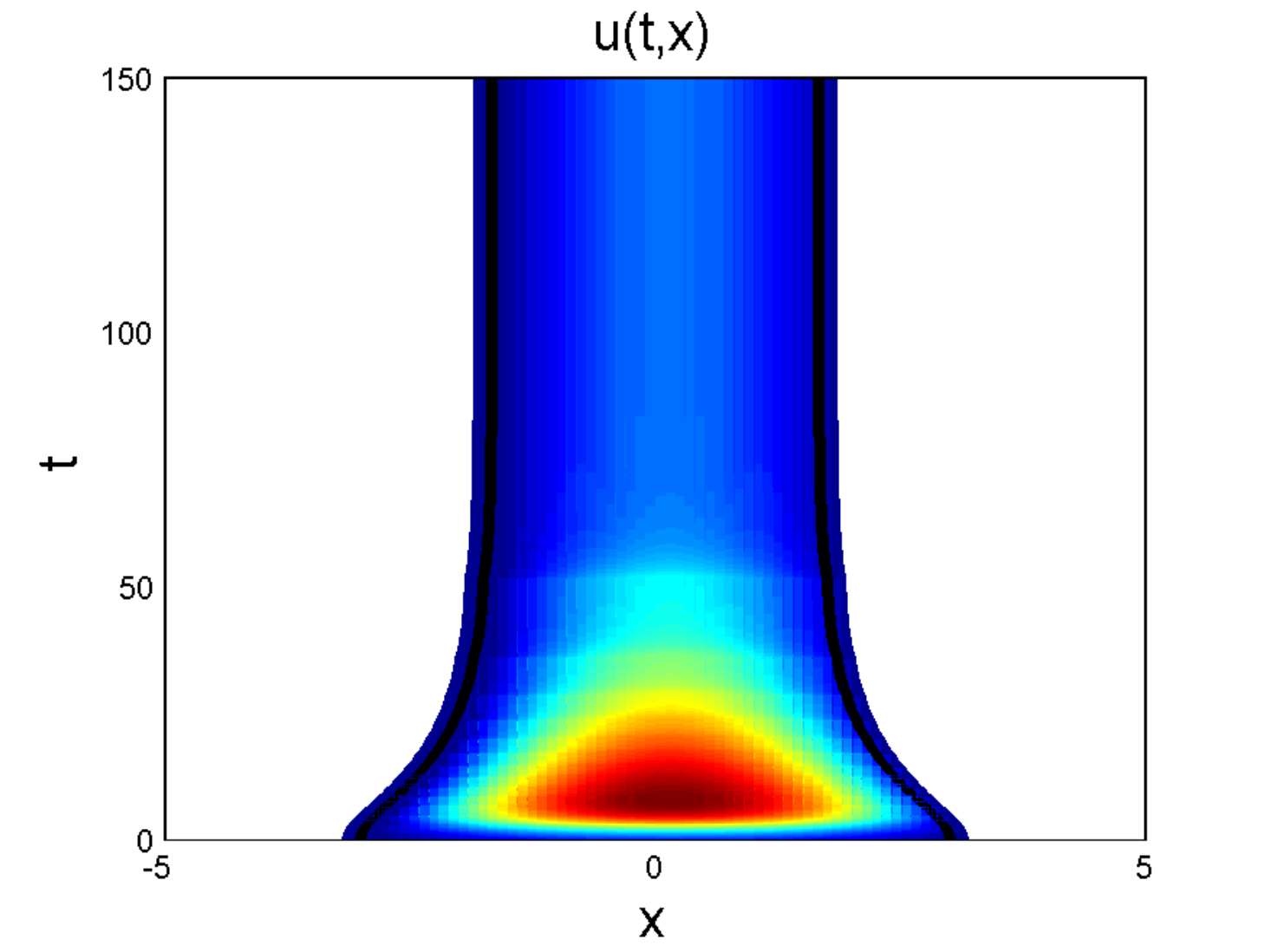}}
\end{minipage}
\begin{minipage}{0.5\linewidth}
\centering\subfigure[$c_1=3.1.$]
{\includegraphics[width=2in]
{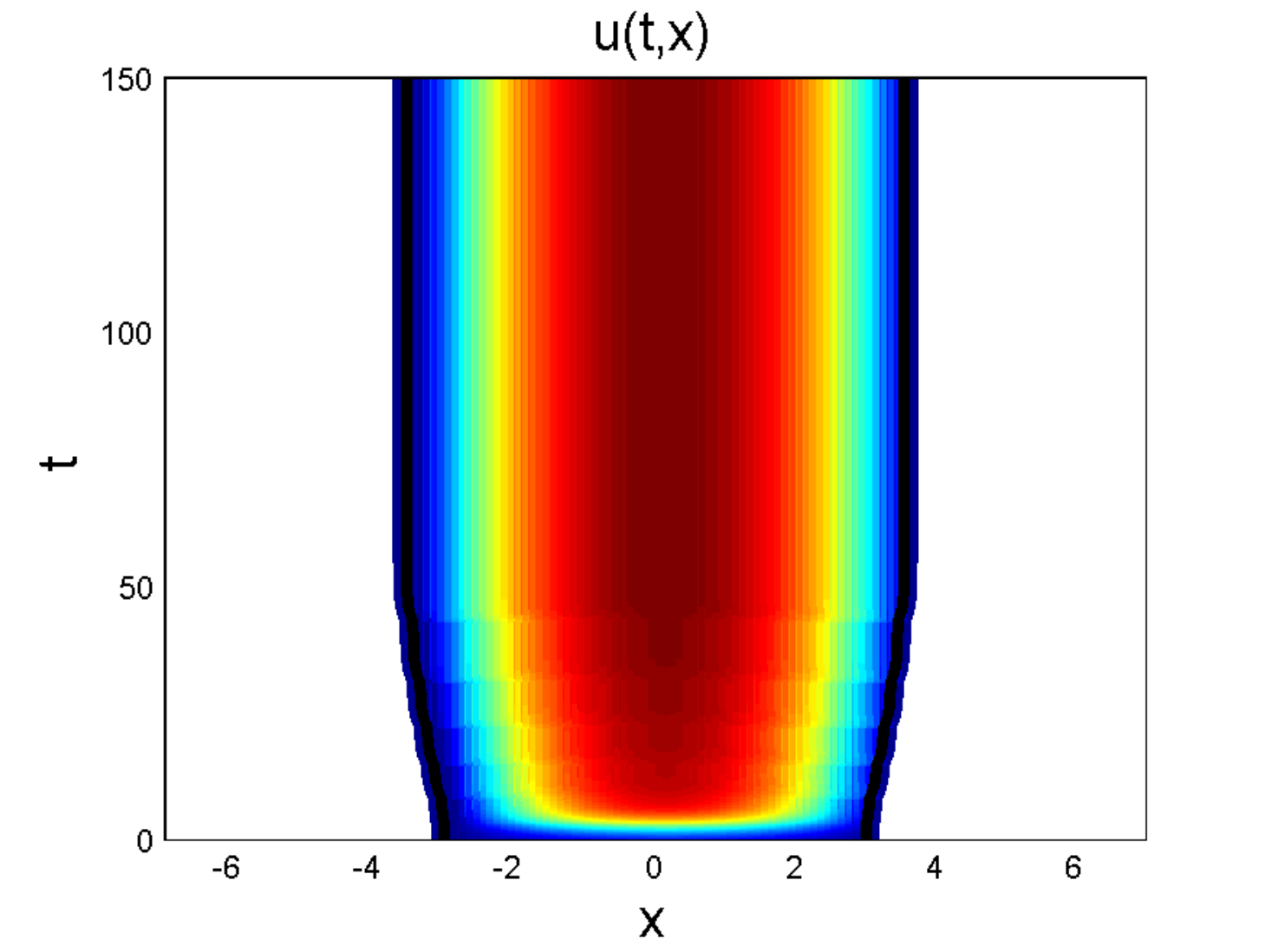}}
\end{minipage}
\begin{minipage}{0.5\linewidth}
\centering\subfigure[$c_1=3.3.$]
{\includegraphics[width=2in]
{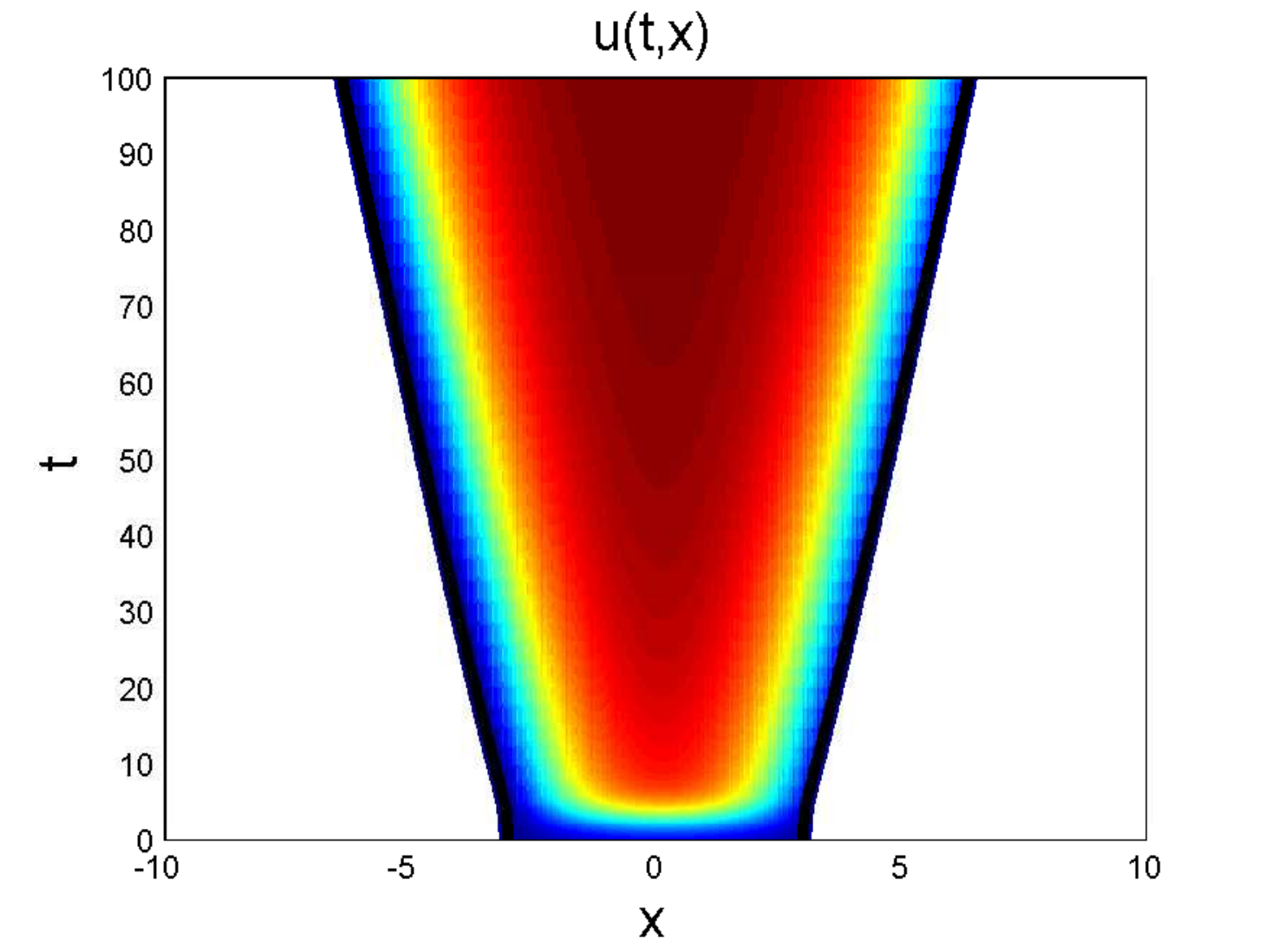}}
\end{minipage}
\caption{The dynamics of \eqref{model2} with different choices of $c_1$. Here, all the other parameters are given by \eqref{par1} and $u_0(x)=0.01(h_0^2-x^2)$.}\label{ex21}
\end{figure}

\begin{figure}[!]
\begin{minipage}{0.5\linewidth}
\centering\subfigure[$c_1=3.5.$]
{\includegraphics[width=2in]
{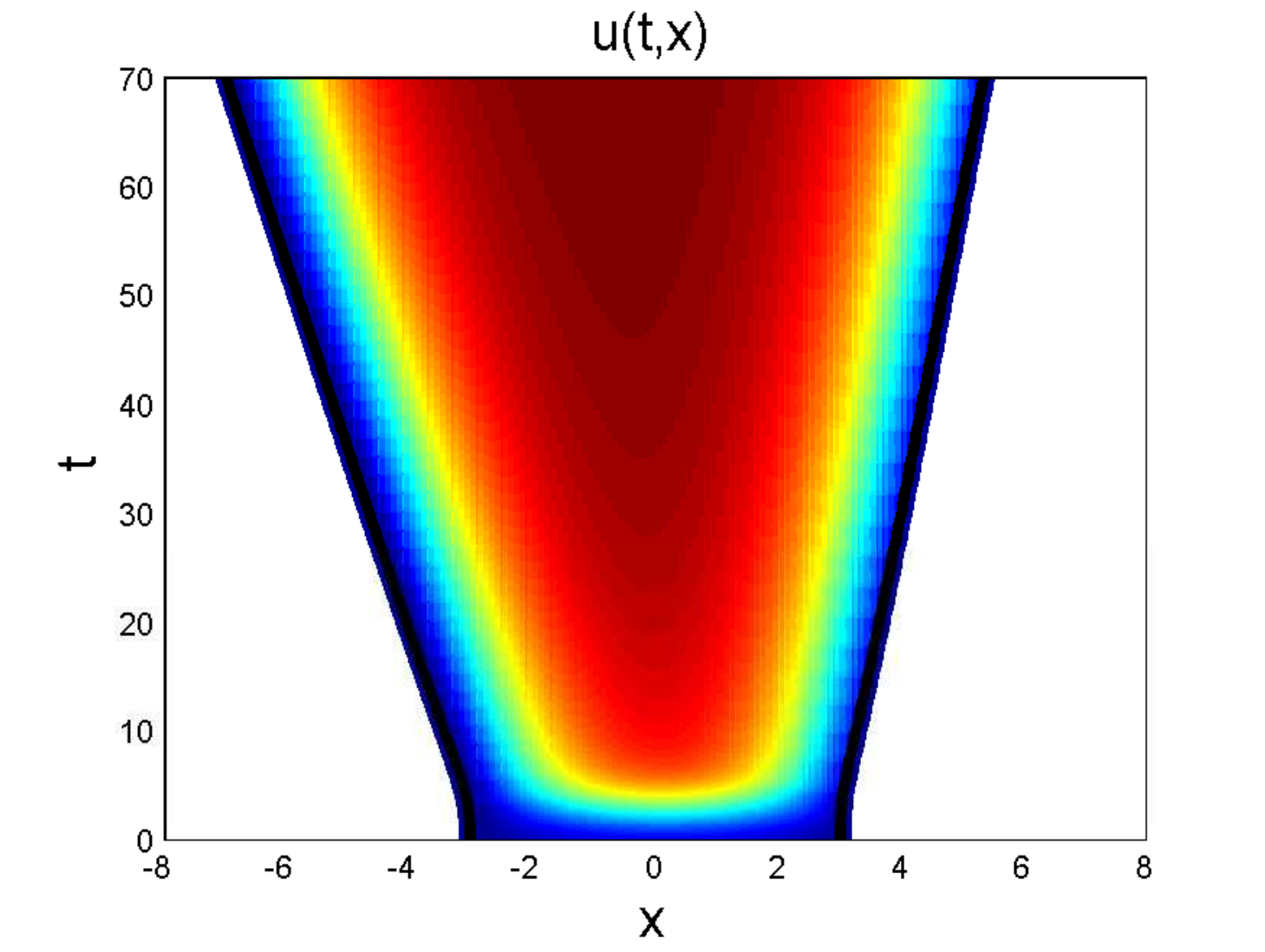}}
\end{minipage}
\begin{minipage}{0.5\linewidth}
\centering\subfigure[$c_1=3.3.$]
{\includegraphics[width=2in]
{figs/c133w.eps}}
\end{minipage}	
\begin{minipage}{0.5\linewidth}
\centering\subfigure[$c_1=3.1.$]
{\includegraphics[width=2in]
{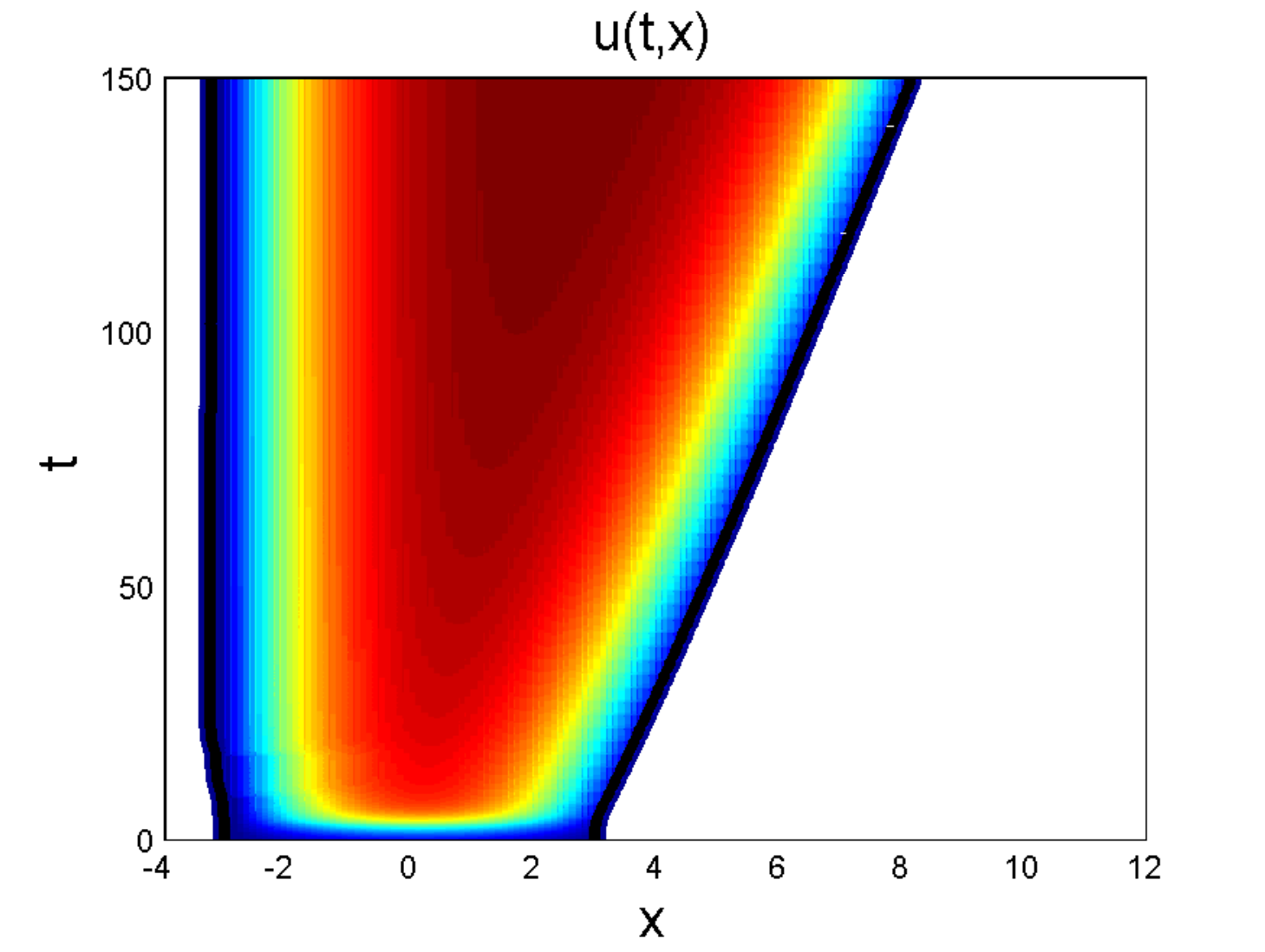}}
\end{minipage}
\begin{minipage}{0.5\linewidth}
\centering\subfigure[$c_1=2.8.$]
{\includegraphics[width=2in]
{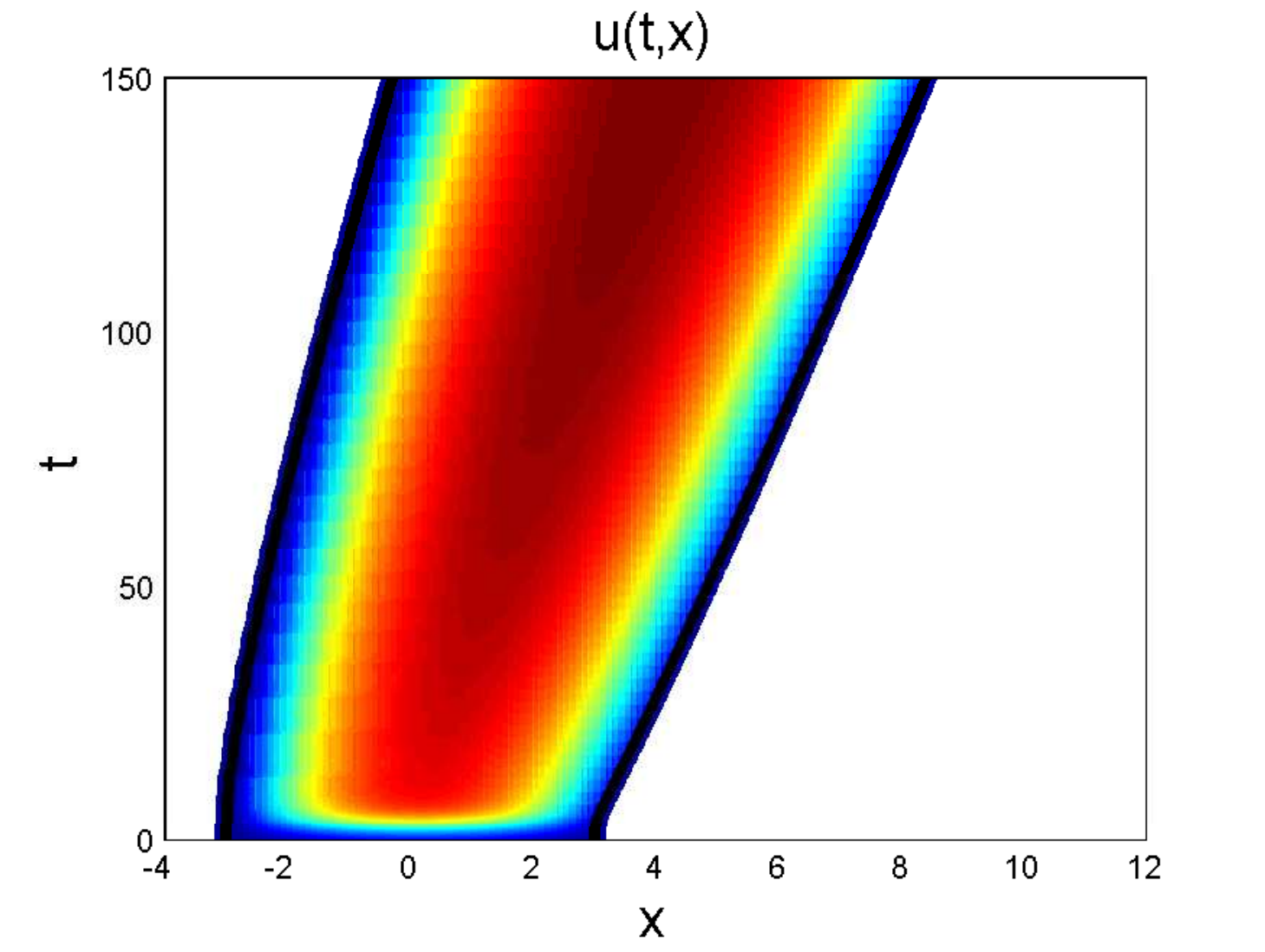}}
\end{minipage}
\begin{minipage}{0.5\linewidth}
\centering\subfigure[$c_1=2.4.$]
{\includegraphics[width=2in]
{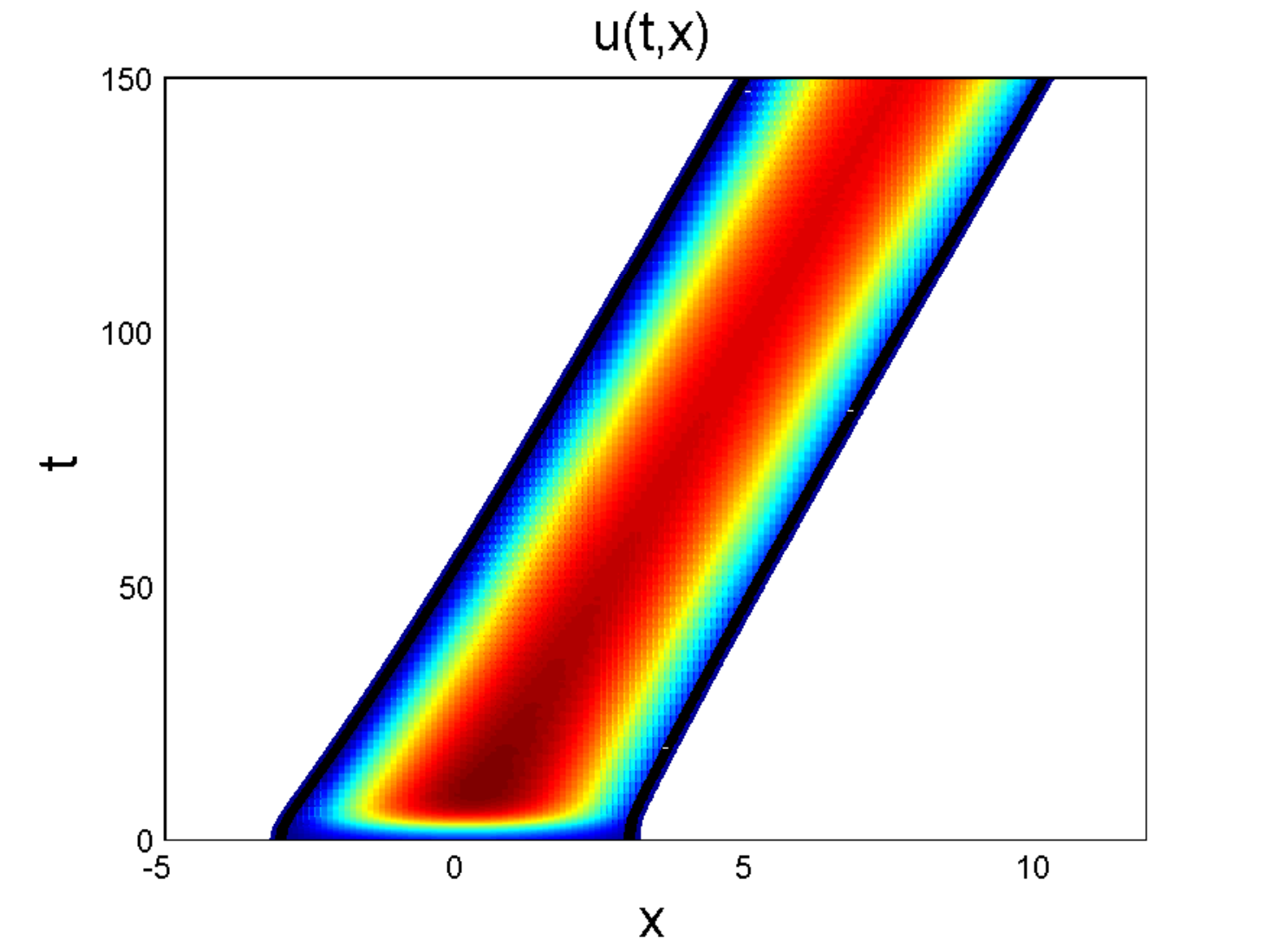}}
\end{minipage}
\begin{minipage}{0.5\linewidth}
\centering\subfigure[$c_1=1.0.$]
{\includegraphics[width=2in]
{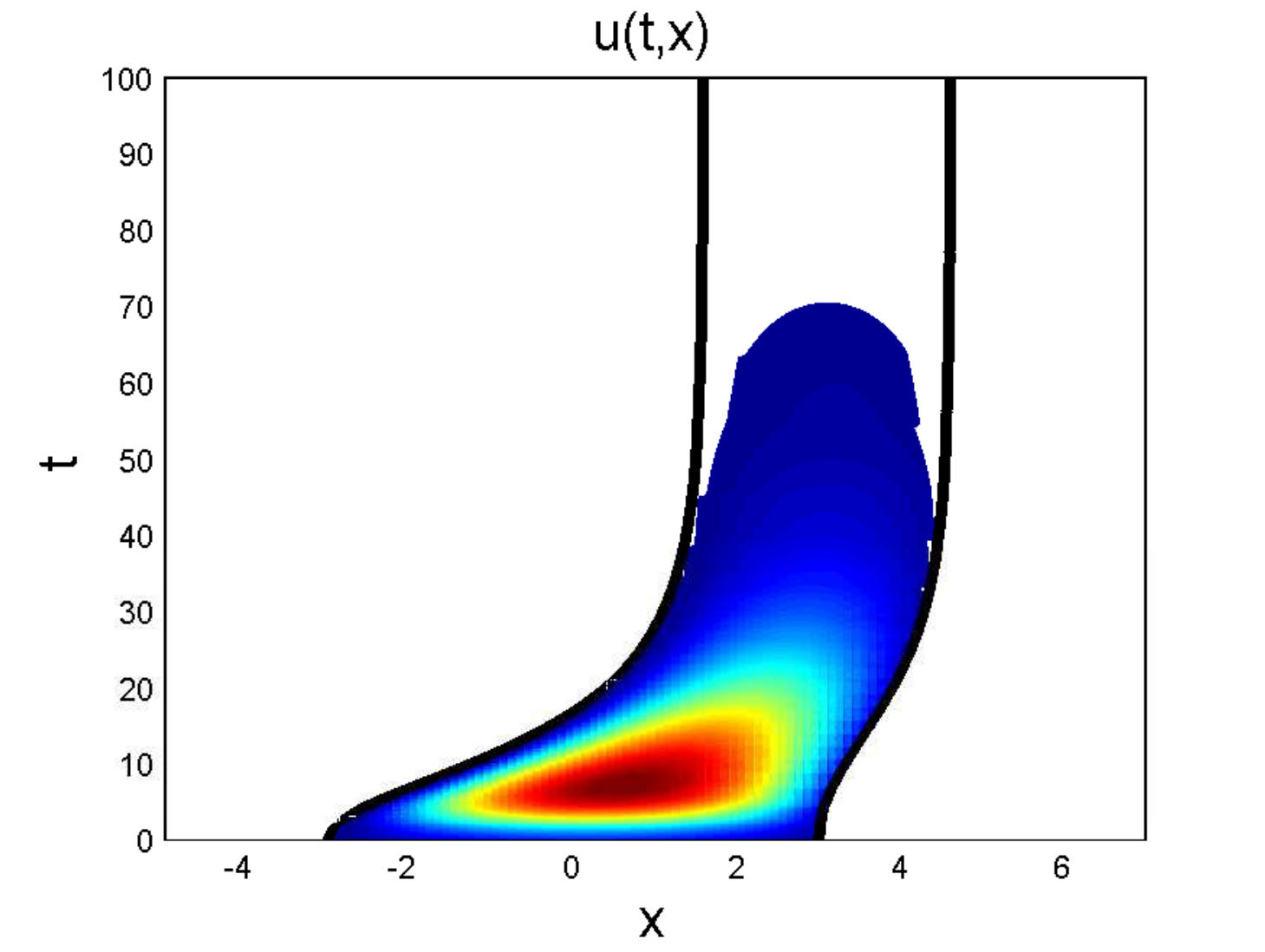}}
\end{minipage}
\caption{The dynamics of \eqref{model2} with different choices of $c_1$. Here, $c_3=3.3$ is fixed, $c_2=c_4=1$, $\alpha_1=\alpha_3=1.9$, $\alpha_2=\alpha_4=1$, $h_0=3$ and $u_0(x)=0.01(h_0^2-x^2)$.}\label{ex22}
\end{figure}

From the results of one boundary problem \eqref{model1}, we know that the increment of $c_1$ will speed up the expansion of the boundary. Thus, for \eqref{model2}, it is expected that the speeds, denoted by $\rho_1$ and $\rho_2$, for the movement of two boundaries might be different, when $c_1\neq c_3$ due to different landscapes on the two free boundaries (all the other parameters are the same as before). To verify this, we set $c_1=3.5$. It is observed that $\rho_2$ will not change, whereas $\rho_1$ increases, indicating that the movement of left boundary is faster than the right boundary, see Figure \ref{ex22} (a). When $c_1=c_3$, we obtain the symmetric case as shown in Figure \ref{ex22} (b). When the value of $c_1$ decreases from $3.3$, the left boundary still moves to the left but with a lower speed. In particular, when $c_1=3.1$, the left boundary does not move (that is, $\rho_1=0$), while the right boundary expands with the same speed $\rho_2$, see Figure \ref{ex22} (c).
As $c_1$ decreases to $2.8$, both left and right boundaries move to the right with $\rho_1<\rho_2$, meaning that the range boundary shifts to the right with expansion as shown in Figure \ref{ex22} (d). When $c_1=2.4$, both boundaries move to the right with a constant range boundary size as shown in Figure \ref{ex22} (e). If we choose $c_1=1$, both boundaries move to the right with a shrinking range boundary size that tends to zero as $t\rightarrow\infty$ as shown in Figure \ref{ex22} (f). In this case, the species cannot survive.

\section{Discussion}\label{discussion}

In this paper, we proposed a more general free boundary problem involving the classic Fisher-KPP reaction-diffusion equation and an integro-differential equation for describing the movement of free boundary of a species' range boundary. The goal of proposing such a phenomenological model is to investigate the nonlocal effect of the total population via weighted functions, on the change of a species' range. We rigorously showed that the Stefan condition is a special case of our free boundary condition. In general, the nonlocal weighted free boundary slows down the spreading speed of the population, which may add biological realism to spreading populations. We also showed that the model possesses more complicated and realistic dynamics than the ones with the Stefan condition such as a steady state solution leading to a spreading-balancing-vanishing trichotomy of the population dynamics, as well as an expanding-balancing-shrinking trichotomy of the range dynamics. Furthermore, more varied spreading scenarios can occur when the free boundaries of the range are allowed to move simultaneously, including both free boundaries spreading in the same direction or uneven spread in two directions.

It should be admitted that many of the interesting results are shown by numerical simulations, and their mathematical proofs may be very challenging. One reason is that we do not know the sign of $h'(t)$ in advance, thus the standard arguments for discussing a free boundary problem with the Stefan condition fail to apply. In addition, the function $w$ is sign-switching, and therefore the comparison theorem is not valid here. This will also bring considerable difficulties for the rigorous analysis. Most mathematical questions have not been resolved yet, and we would like to propose them as open questions for future study.

There are also several natural variations of our free boundary model formulation, which would be interesting to investigate.  For example, as described in Introduction, it would be intriguing to switch the Dirichlet boundary condition at the range boundary ($u(t,h(t))=0$) to a Robin (mixed) or Neumann boundary condition to see the effect on the range dynamics.  Furthermore, the assumption that individuals immediately die when leaving the range could be relaxed to include the possibility that they die with a given fixed per capita mortality rate when outside the range boundary.  Further extensions could include species interactions, such as predation, competition, and symbiosis. In a multi-species model with free boundary, the interaction may also affect the movement equation of a free boundary. To incorporate this realism into the free boundary model will be a significant contribution to the species' movement modeling with free boundary.

Gueron and Levin \cite{shay} showed a finger-shaped pattern for a large wildebeest herd. This pattern varies over huge scales that are far larger than species' perceptual ranges, therefore random fluctuations cannot explain this phenomenon. Our free boundary model proposed in the paper has potential to answer this question from a new perspective. It is necessary to derive the new free boundary problem in a two-dimensional space, and then we need to work on a specific species that allows parametrization.

In this paper we did not include resource density. However, an extension of the model could include renewable resources, with resource limitation included explicitly in the model via a resource equation to describe resource dynamics. This resource-explicit model will be more mechanistic and could better describe the change of free boundaries for a species' range. Other considerations include heterogeneous environment and seasonality. We leave all these future directions as open problems.

\section*{Acknowledgments}

Hao Wang is partially supported by NSERC Individual Discovery Grant RGPIN-2020-03911 and NSERC Discovery Accelerator Supplement Award RGPAS-2020-00090.
Chuncheng Wang is partially supported by NSFC No. 12171118 and Heilongjiang NSF No. LH2019A010.
Mark Lewis is partially supported by a Canada Research Chair and an NSERC Grant.
Chunxi Feng is partially supported by a CSC Doctoral Scholarship.

\section*{Appendix}
Here we prove \eqref{tostefan}. Let
$$ H_1(x)=
\begin{cases}
1,&~~x\geq0\\
0,&~~x<0
\end{cases}$$
be the Heaviside function and $f$ be a smooth test function which is bounded and integrable. Define
$$
S_n(x)=n(e^{-nx}-e^{-n^2x})H_1(x),~~~~n\in\mathbb{N}.
$$
\begin{lemma}\label{Sn}
$S_n(x)$ is a delta sequence, in the sense that
\begin{equation}\label{delta}
\lim_{n\rightarrow\infty}\int_{-\infty}^{\infty}f(x)S_n(x)dx=f(0).
\end{equation}
\end{lemma}

\begin{proof}
From $\int_{-\infty}^{\infty}S_n(x)dx
=\int_0^{\infty}S_n(x)dx=1-\frac{1}{n}$, we have $\lim_{n\rightarrow\infty}\int_{-\infty}^{\infty}S_n(x)dx=1$.
Let
$$
\lim_{n\rightarrow\infty}\int_{-\infty}^{\infty}f(x)S_n(x)dx= \lim_{n\rightarrow\infty}\int_{-\infty}^{\infty}f(0)S_n(x)dx+ \lim_{n\rightarrow\infty}\int_{-\infty}^{\infty}g(x)S_n(x)dx,
$$
where $g(x)=f(x)-f(0)$.
The equation \eqref{delta} can be proved if we show
\begin{equation}\label{F}
\lim_{n\rightarrow\infty}\int_{-\infty}^{\infty}g(x)S_n(x)dx=0.
\end{equation}
To this end, we will demonstrate that for any $\varepsilon>0$, there exists $N$ such that
\begin{equation}\label{G}
\left|\int_{-\infty}^{\infty}g(x)S_n(x)dx\right|<\varepsilon,
~~~~~~\forall n>N.
\end{equation}
Given $\delta>0$, we write
\begin{equation}\nonumber
\left|\int_{-\infty}^{\infty}g(x)S_n(x)dx\right|
\leq\left|\int_0^\delta g(x)S_n(x)dx\right|+\left|\int_\delta^{\infty}g(x)S_n(x)dx\right|
=:I_1+I_2
\end{equation}
Since $g(0)=0$ and $g(x)$ is continuous at $x=0$, we know $\lim_{\delta\rightarrow0^+}M_1(\delta)=0$, where $M_1(\delta):=\max_{0\leq x\leq\delta}|g(x)|$. Note that $0<\int_0^\delta S_n(x)dx<1$ for $n>1$. Therefore, for any $\varepsilon>0$, there exists $\delta>0$, independent of $n>1$, such that
\begin{equation}\label{I1}
I_1=\left|\int_0^\delta g(x)S_n(x)dx\right|\leq M_1(\delta)\int_0^\delta S_n(x)dx<\frac{\varepsilon}{2}.
\end{equation}
For $I_2$, we have
$$
I_2=\left|\int_\delta^{\infty}g(x)S_n(x)dx\right|\leq M_2(\delta)\int_\delta^{\infty}S_n(x)dx,
$$
where $M_2(\delta)=\sup_{x>\delta}|g(x)|$. Let $N=N(\delta)=\frac{1}{\delta}\ln\frac{2M_2(\delta)}{\varepsilon}$. Then, for $n>N$, we have
\begin{equation}\nonumber
I_2\leq M_2(\delta)\int_\delta^{\infty}n(e^{-nx}-e^{-n^2x})dx
=M_2(\delta)\left(e^{-n\delta}-\frac{e^{-n^2\delta}}{n}\right)
<M_2(\delta)e^{-n\delta}<\frac{\varepsilon}{2}.
\end{equation}
This, together with \eqref{I1}, proves \eqref{G}.
\end{proof}

\begin{corollary}\label{wn}
Let
\begin{equation}\label{H}
w_n(x)=(n^3e^{-n^2x}-n^2e^{-nx})H_1(x).
\end{equation}
Then
$$
\lim_{n\rightarrow\infty}\int_{0}^{h}w_n(x)f(h-x)dx
=-f'(h),
$$
where $f$ is a test function as defined above.
\end{corollary}
\begin{proof}
We start by defining a remainder term:
\begin{displaymath}
R_n=\lim_{n\rightarrow\infty}\int_h^{\infty}w_n(x)f(h-x)dx
\end{displaymath}
for any $h>h_0>0$ and show that $\lim_{n\to\infty}R_n=0$.  Consider
\begin{equation}
\begin{aligned}
\lim_{n\to\infty}|R_n|&=\lim_{n\to\infty}\left|\int_h^{\infty}w_n(x)f(h-x)dx\right|
&\leq \lim_{n\to\infty}M_2(h) \left|n(e^{-n^2h}-e^{-nh})\right|=0,
\end{aligned}
\end{equation}
where $M_2$ is defined in a manner analogous to that in the proof of Lemma \ref{Sn}, and thus we have that  $\lim_{n\to\infty}R_n=0$ also.

Since $S_n'(x)=w_n(x)$  it follows from Lemma \ref{Sn} that
\begin{equation}\nonumber
\begin{aligned}
\lim_{n\rightarrow\infty}\int_{0}^{h}w_n(x)f(h-x)dx&
=\lim_{n\rightarrow\infty}\int_{0}^{\infty}w_n(x)f(h-x)dx-\lim_{n\rightarrow\infty}\int_h^{\infty}w_n(x)f(h-x)dx\\
&=\lim_{n\rightarrow\infty}\int_{0}^{\infty}S_n'(x)f(h-x)dx\\
&=\lim_{n\rightarrow\infty}[S_n(x)f(h-x)]_{0}^\infty
-\lim_{n\rightarrow\infty}\int_{0}^{\infty}S_n(x)f'(h-x)dx\\
&=-\lim_{n\rightarrow\infty}\int_{-\infty}^{\infty}S_n(x)f'(h-x)dx\\
&=-f'(h),
\end{aligned}
\end{equation}
where we have used $S_n(x)=0$ for $x\leq 0$ and $\lim_{x\to\infty} S_n(x)=0$.
\end{proof}

Then equation \eqref{tostefan} is the direct consequence of Corollary \ref{wn} by setting $c_1=n^3,~c_2=n^2,~\alpha_1=n^2,~\alpha_2=n ~$ in \eqref{w}.

Consider
\begin{equation}\label{appro}
\begin{cases}
-u_{xx}=u(a-u),~~0<x<h,\\
u_x(0)=0,~~u(h)=0,
\end{cases}
\end{equation}
for $a>0$.
\begin{theorem}\label{est}
For any given $\epsilon>0$, there exists sufficiently large $a$ such that \eqref{appro} admits a unique spatially inhomogeneous solution $\hat{u}(x)$, satisfying
\begin{equation}\label{hatu1}
\hat{u}(x)=a-g_a(x),~~~~~~~~x\in[0,h-\epsilon],
\end{equation}
where $g_a(x)\geq0$ is a function of $x$, such that
$$
\frac{g_a(x)}{a}\rightarrow0,~~~~~~~~x\in[0,h-\epsilon],
$$
as $a\rightarrow\infty$.
\end{theorem}

\begin{proof}
Let $v(x)=u(-x)$ for $x\in[-h,0]$. Then,
$$
-v_{xx}=v(a-v),~~~~v(-h)=v_x(0)=0.
$$
Therefore, if $w$ is the solution of
\begin{equation}\label{wequ}
\begin{cases}
-w_{xx}=w(a-w),~~-h<x<h,\\
w(-h)=0,~~w(h)=0,
\end{cases}
\end{equation}
Then, the restriction of $w$ on $[0,h]$ is the solution of \eqref{appro}. Let $w_a(x)$ be the solution of \eqref{wequ}. By Theorem $3.9$ in \cite{Wang2021}, it follows that, for any compact set $K$ in $[-h,h]$,
$$
\frac{w_a(x)}{a}\rightarrow 1,~~~~\text{as}~a\rightarrow \infty
$$
This means, for any given $\epsilon>0$, there exists sufficiently large $a$ and a function $g_a(x)$, such that \eqref{hatu1} is satisfied.
\end{proof}

\bibliographystyle{plain}
\bibliography{FB}

\end{document}